\documentclass[12pt, a4paper]{amsart}
\usepackage[english]{babel}
\usepackage{amssymb,amsmath,amsthm,mathrsfs}
\usepackage{verbatim}
\usepackage{graphicx,wrapfig}
\usepackage[margin=1in]{geometry}
\usepackage{stmaryrd}
\usepackage{hyperref}
\hypersetup{
    colorlinks,
    citecolor=black,
    filecolor=black,
    linkcolor=black,
    urlcolor=black
}
\usepackage{multicol}

\newtheorem{theorem}{Theorem}
\newtheorem{example}[theorem]{Example}
\newtheorem{problem}[theorem]{Problem}
\newtheorem{exercise}[theorem]{Exercise}

\newtheorem{definition}[theorem]{Definition}
\newtheorem{lemma}[theorem]{Lemma}
\newtheorem{remark}[theorem]{Remark}
\newtheorem{proposition}[theorem]{Proposition}
\newtheorem{corollary}[theorem]{Corollary}

\newtheorem*{claim*}{Claim}
\newtheorem*{theorem*}{Theorem}

\def\bal{\begin{aligned}}
\def\eal{\end{aligned}}
\def\be{\begin{equation}\label}
\def\ee{\end{equation}}
\def\bcs{\begin{cases}}
\def\ecs{\end{cases}}
\def\={\;=\;}
\def\+{\,+\,}
\def\-{\,-\,}
\def\C{{\mathbb C}}
\def\Z{{\mathbb Z}}

\def\Q{{\mathbb Q}}
\def\R{{\mathbb R}}
\def\F{{\mathbb F}}
\def\P{{\mathbb P}}
\def\A{{\mathbb A}}
\def\T{{\mathbb T}}

\def\supp{{\rm supp}}
\def\lb{\llbracket}
\def\rb{\rrbracket}
\def\ord{\mathrm{ord}}

\def\sF{\mathcal{F}}
\def\sG{\mathcal{G}}
\def\fil{\mathcal{F}}
\def\sP{\mathcal{P}}

\def\cartier{\mathscr{C}_p}

\def\v#1{{\bf #1}}

\def\is{\equiv}
\def\mod#1{({\rm mod}\ #1)}
\def\hat{\widehat}

\title{Congruences and cohomology}
\author{Masha Vlasenko}\thanks{This work is supported by the National Science Centre of 
Poland (NCN), grant UMO-2020/39/B/ST1/00940.}
\email{masha.vlasenko@gmail.com}
\address{Kyiv School of Economics, Mykoly Shpaka 3, 03113 Kyiv, Ukraine }
\address{Institute of Mathematics of the Polish Academy of Sciences, Sniadeckich 8, 00-656 Warsaw, Poland}

\begin{document}
\maketitle

\bigskip

These are notes of my lecture courses given in the summer of 2024 in the School on Number Theory and Physics at ICTP in Trieste and in the 27th Brazilian Algebra Meeting at IME-USP in S\~ao Paulo. We give an elementary account of $p$-adic methods in de Rham cohomology of algebraic hypersurfaces with explicit examples and applications in number theory and combinatorics.  These lectures are based on the series of our joint papers with Frits Beukers entitled \emph{Dwork crystals} (\cite{DCI,DCII,DCIII}). These methods also have applications in mathematical physics and arithmetic geometry (\cite{IN,Cartier0}), which we overview here towards the end. I am grateful to the organisers of both schools and to the participants of my courses whose questions stimulated writing these notes. 

\smallskip

\emph{I would like to dedicate these lectures to the defenders of Ukraine, with gratitude for the possibility to live and work at the time when our country is being under constant military attacks. }

\tableofcontents

\section{Counting points on algebraic varieties over finite fields}

For an algebraic variety $X$ over a finite field $\F_p$ there exist non-negative integers $m,k$ and algebraic numbers $\alpha_1,\ldots,\alpha_m, \beta_1, \ldots, \beta_k \in \overline{\Q}$ such that numbers of points on $X$ over all field extensions are given by
\be{Frob-roots}
\# X(\F_{p^s}) = \sum_{i=1}^m \alpha_i^s - \sum_{j=1}^k \beta_j^s \quad \text{ for all } \quad s \ge 1. 
\ee
We will refer to $\alpha_i$'s and $\beta_j$'s as the \emph{Frobenius roots} of $X$. Their existence was proved by Bernard Dwork circa 1960 using $p$-adic methods. The natural building blocks in Dwork's approach were hypersurfaces. We are going to elaborate explicit $p$-adic formulas for Frobenius roots of hypersurfaces and explore related algebraic constructions. 

\subsection{Affine and projective algebraic varieties} 

An \emph{affine algebraic variety} $X$ over a field $k$ is a common zero locus of a collection of multivariate polynomials $f_1(\v x),\ldots,f_m(\v x) \in k[x_1,\ldots,x_n]$. Sets of points of $X$ over any field extension  $K \supseteq k$ are defined as
\[
X(K) = \{ \v x=(x_1,\ldots,x_n) \in K^n : f_1(\v x) = \ldots = f_m(\v x)=0 \}.
\]
If there are no equations ($m=0$) we denote $X=\mathbb{A}^n$, the $n$-dimensional \emph{affine space}. In this case one simply has $\mathbb{A}^n(K) = K^n$. 

More general algebraic varieties are glued of affine pieces. A straightforward  example of such a glueing is given by projective spaces and projective varieties. The $n$-dimensional \emph{projective space} $\mathbb{P}^n$ is the variety of lines through the origin in $\mathbb{A}^{n+1}$. Its sets of points are given by 
\[
\mathbb{P}^n(K) = ( K^{n+1} \setminus\{\v 0\} ) / K^{\times} = \{ (x_0,\ldots,x_n) \in K^{n+1} : \exists i \quad x_i \ne 0 \} / \sim,   
\]
where $\sim$ denotes the homothety equivalence $(x_0,\ldots,x_n) \sim (c x_0,\ldots,c x_n)$ for all $c \in K^{\times}$. The respective equivalence class is denoted by $[x_0:\ldots:x_n]$. These brackets are called the homogeneous coordinates on $\mathbb{P}^n$. The $n$-dimensional projective space can be covered by $n+1$ affine spaces
\[\bal
\mathbb{P}^n = \cup_{i=0}^{n} U_i, \qquad U_i = \{ [x_0:\ldots:x_n] : x_i \ne 0\} &\cong \mathbb{A}^n, \\
[x_0:\ldots:x_n] & \mapsto (\frac{x_0}{x_i},\ldots,\frac{x_{i-1}}{x_i},\frac{x_{i+1}}{x_i},\ldots,\frac{x_n}{x_i}). 
\eal\] 
A \emph{projective variety} $X \subset \mathbb{P}^n$ is a common zero locus of a collection of homogeneous polynomials $F_1(\v x),\ldots,F_m(\v x) \in k[x_0,\ldots,x_n]$. It is a union $X = \cup_{i=0}^n X_i$ of affine varieties  $X_i = X \cap U_i$. Here $X_i$ can be identified with the zero locus of polynomials in $n$ variables given by
$f_j(x_0,\ldots,x_{i-1},x_{i+1},\ldots,x_n) = F_j(x_0,\ldots,x_{i-1},1,x_{i+1},\ldots,x_n)$ for $j=1,\ldots,m$.
   
\begin{exercise}\label{ell-curve-exercise} Let $a,b,c \in k$. The projective cubic curve
\[
x_0 x_1^2 = x_2^3 + a x_0 x_2^2 + b x_0^2 x_2 + c x_0^3 
\]
is given in the affine sheet $U_0 \subset \P^2$ by the equation $y^2 = x^3 + a x^2 + b x +c$ in coordinates $(x,y)=(\frac{x_2}{x_0},\frac{x_1}{x_0})$. Check that the only point on this curve in $\P^2 \setminus U_0$ is $O=[0:1:0]$. Look up the definition of non-singular algebraic varieties and check that this curve is non-singular whenever the polynomial $x^3+a x^2 + b x + c$ has no multiple roots. If is the case, this curve is called an elliptic curve. 
\end{exercise}

Non-singular cubic curves with a specified point $O$ defined over $k$ are called \emph{elliptic curves} over $k$. When the characteristic of $k$ is not equal to $2$ or $3$ such a curve can be always given by an equation as above in a suitably chosen system of coordinates. 

\subsection{Weil conjectures} \label{sec:weil-conj}

Let $p$ be a prime number and $X$ be an algebraic variety over the finite field with $p$ elements $k=\F_p$.
The natural generating series for the numbers of points $\#X(\F_{p^s})$ turns out to be 
\[
\mathcal{Z}_X( T ) = \exp \left( \sum_{s=1}^\infty \#X(\F_{p^s}) \frac{T^s}{s}\right) \in \Q\lb T\rb.
\]
This formal series is called the \emph{local zeta function} of $X$. The naturality of the above definition of a zeta function is explained by the following fact.

\begin{theorem}[Dwork,~\cite{Dw-rationality}]\label{rationality-thm} The local zeta function of an algebraic variety\footnote{In this theorem algebraic varieties are assumed to be of \emph{finite type}, which means they can be covered by finite number of affine pieces. It is a standard assumption in the definition of varieties.} $X$ over $\F_p$ is rational:
\[
\mathcal{Z}_X( T ) \in \Q(T).
\]
\end{theorem}

Writing this rational function as $\prod_{j=1}^k (1 - \beta_j T) / \prod_{i=1}^m (1 - \alpha_i T)$ we obtain expression~\eqref{Frob-roots} for the numbers of points on $X$ over the extensions of $\F_p$.

The statement of Theorem~\ref{rationality-thm} was the first one on the list of properties of local zeta functions which were conjectured by Andr\'e Weil in 1940s. The rest of Weil's conjectures concerns non-singular projective varieties $X$. If $X$ is such a variety of dimension~$n$ then
\[
\mathcal{Z}_X(T) = \prod_{i=0}^{2n} P_i(T)^{(-1)^{i-1}} = \frac{P_1(T) \ldots P_{2n-1}(T)}{P_0(T) \ldots P_{2n}(T)}
\]
where $P_i \in 1 + T\Z[T]$ are integral polynomials with the following properties:
\begin{itemize}
\item[(RH)] $P_0(T)=1-T$, $P_{2n}(T)=1-p^n T$ and for $1 \le i \le 2n-1$ polynomials $P_i(T)$ factor over $\C$ as $ P_i(T) = \prod_{j=1}^{\beta_i}(1-\alpha_{i,j}T)$ with $\alpha_{i,j} \in \overline{\Q}$ satisfying $|\alpha_{ij}|=p^{i/2}$.
\item[(FE)] $\mathcal{Z}_X(\frac{1}{p^n T}) = \pm (p^{n/2}T)^{\chi(X)} \mathcal{Z}_X(T)$ where $\chi(X)$ is the Euler characteristic of $X$; in particular, for every $\alpha_{i,j}$ we have that $p^n/\alpha_{i,j}$ is a reciprocal root of $P_{2n-i}(T)$.
\item[(TOP)] If $X$ can be lifted to a smooth variety defined over a number field then $\deg P_i(T)$ equals to the $i$th Betti number of the topological space of complex points of this variety.
\end{itemize}
Everything which was said above also holds true for varieties over finite fields $\F_q$ where $q$ is a power of $p$, and one should substitute $p$ with $q$ throughout the statement of the Weil conjectures. We restrict to the case of $\F_p$ for brevity. These conjectures were proved in 1960s and 1970s leading to the creation of $p$-adic and \'etale cohomology theories, with contributions by Bernard Dwork, Alexander Grothendieck, Michael Artin, Pierre Deligne and many other mathematicians.

\begin{exercise} Compute the local zeta functions of $\A^n$ and $\P^n$. 
\end{exercise}

\begin{exercise} Consider a non-singular projective conic (that is, a curve of degree two) $C \subset \P^2$. Assuming $C$ has a point over $\F_p$, show that its zeta function is given by $\frac{1}{(1-T)(1-pT)}$. 

For this purpose you may fill the details in the following argument over a field $K=\F_{p^s}$: every line through our specified point (defined over $k=\F_p$) will have one other point of intersection with the conic; the line is defined over $K$ if and only if this other point is defined over $K$; hence $\# C(K) = \#\P^1(K)=p^s+1$.    
\end{exercise}

Computation of zeta functions of cubic curves is more sophisticated. The zeta function of an elliptic curve $E \subset \P^2$ defined over $\F_p$ is given by
\be{ell-curve-zeta}
\mathcal{Z}_E( T ) = \frac{1-a_p T + pT^2}{(1-T)(1-pT)},
\ee
where the number $a_p = p+1 - \# E(\F_p) \in \Z$ is called the \emph{Frobenius trace} of $E$.  Helmut Hasse  proved that $|a_p| \le 2 \sqrt{p}$. Note that this implies that the discriminant of the quadratic polynomial in the denominator is $\le 0$. If the discriminant could be $0$ we would have $a_p=2 \sqrt p \notin \Z$. Hence the two reciprocal roots are complex conjugate and have absolute value $\sqrt p$ in agreement with (RH).   
The Sato--Tate conjecture is a statement about the distribution of numbers $a_p/(2\sqrt{p})$ in the interval $[-1,1]$ when $p$ varies and $E$ is an elliptic curve without complex multiplication defined over $\Q$. This conjecture was proved by Laurent Clozel, Michael Harris, Nicholas Shepherd-Barron, and Richard Taylor circa 2008 under mild assumptions. The proof was completed by Thomas Barnet-Lamb, David Geraghty, Harris, and Taylor in 2011. Several generalizations to other algebraic varieties are open. We wanted to mention this to motivate our interest in the knowledge of Frobenius roots.

\subsection{The role of hypersurfaces in proving rationality of local zeta functions}

Dwork's proof of Theorem~\ref{rationality-thm} used $p$-adic analysis. The building blocks of this proof are algebraic varieties given by one equation. Such varieties are called \emph{hypersurfaces}.  

\begin{proposition} It suffices to prove Theorem~\ref{rationality-thm} for affine hypersurfaces.
\end{proposition}

\begin{proof} An algebraic variety can be covered by a finite number of affine pieces whose intersections are also affine.\footnote{An algebraic variety $X$ is of finite type if it can be covered by a finite number of affine pieces. Under an additional assumption that $X$ is separated, this covering can be refined to satisfy the property that intersections are also affine. Here we give a proof for separated varieties. Without this assumption one needs to proceed in two steps. First, we deduce rationality of zeta functions for separated varieties. Second, we note that intersections of affine varieties are separated and apply the inclusion-exclusion argument using separated varieties instead of hypersurfaces.} If $X=\cup_{i=1}^m X_i$ is such a covering then the inclusion-exclusion principle reduces the count of points on $X$ to the count of points on its affine pieces:
\be{incl-excl}
\# (\cup_{i=1}^m X_i) (\F_{p^s}) = \sum_{k=1}^m (-1)^{k-1} \sum_{1 \le i_1 \le \ldots \le i_k \le m} \#( X_{i_1} \cap \ldots \cap X_{i_k} )(\F_{p^s}).   
\ee   

Now let $X$ be an affine variety of common zeroes of polynomials $f_1(\v x),\ldots,f_m(\v x) \in \F_p[x_1,\ldots,x_n]$. For a polynomial $f \in \F_p [x_1,\ldots,x_n]$ we denote the hypersurface of its zeroes by
\[
X_f = \{ \v x: f(\v x)=0 \}. 
\]
Then $X = X_{f_1} \cap X_{f_2} \cap \ldots \cap X_{f_m}$. We also note that the union of hypersurfaces is again a hypersurface: $X_{f_1} \cup X_{f_2} \cup \ldots \cup X_{f_m} = X_{f_1\cdot \ldots \cdot f_m}$. Let us apply inclusion-exclusion formula~\eqref{incl-excl} with $X_i = X_{f_i}$. Since in the left-hand side of~\eqref{incl-excl} we have a hypersurface, this formula allows to do induction on the number $m$ of equations. When $m=2$ one has
\[
\# (X_{f_1 \cdot f_2}) (\F_{p^s}) = \# X_{f_1}(\F_{p^s}) + \# X_{f_1}(\F_{p^s}) - \#( X_{f_1} \cap X_{f_2} )(\F_{p^s}),
\]  
which expresses the zeta function of an intersection of two hypersurfaces as a ratio of zeta functions of single hypersurfaces. Similarly, formula~\eqref{incl-excl} expresses zeta function of an intersection of $m$ hypersurfaces (the last term in the right-hand side) as a ratio of zeta functions of varieties given by less than $m$ equations.   
\end{proof}   

In the above proof one can also see that having an effective method of computing Frobenius roots for hypersurfaces would be sufficient, as Frobenius roots of general varieties occur among the roots of their affine pieces. Of course some cancellations would happen in the multiplication and division of zeta functions of pieces. 

\subsection{Weil cohomology and de Rham cohomology}

Andr\'e Weil suggested that his conjectures would follow from the existence of a suitable cohomology theory for varieties over finite fields, similar to the usual (singular) cohomology for complex varieties. For a variety $X$ over $\F_p$ there is the $p$th power Frobenius map on the set $X(\overline\F_p)$ of its points defined over all extensions of the ground field. Here $\overline \F_p$ denotes the algebraic closure of $\F_p$. The Frobenius map is expressed locally as raising coordinates of points to the $p$th power: $(x_1,\ldots,x_n) \mapsto (x_1^p,\ldots,x_n^p)$. Since $\F_{p^s} \subset \overline \F_p$ can be identified as the set of solutions to the eqiation $x^{p^s}=x$, we see that $X(\F_{p^s}) \subset X(\overline\F_p)$ is the set of fixed points of the $s$th power of the Frobenius map for every $s \ge 1$. In algebraic topology the number of fixed points of a continuous map can be worked out using the Lefschetz fixed-point theorem, given as an alternating sum of traces of the induced map on its cohomology groups. Weil anticipated existence of a cohomology theory which would asoociate to a variety over a finite field vector spaces over a field $F$ of characteristic zero (cohomology groups) along with a canonical linear map on them, traces of powers of which would count points over the extensions of the ground field. Such a map would be also called the Frobenius map, and we called the $\alpha_i$'s and $\beta_j$'s in formula~\eqref{Frob-roots} the Frobenius roots referring to this hypothetical Frobenius map. One could also call them the Frobenius eigenvalues, or simply the reciprocal roots of the zeta function.

As was mentioned in \S\ref{sec:weil-conj}, such cohomology theories were constructed in subsequent decades with coefficients in $F=\Q_\ell$, $\ell \ne p$ ($\ell$-adic or \'etale cohomology)\footnote{ There is also \'etale cohomology with $\ell=p$.} and $F=\Q_p$ ($p$-adic or rigid cohomology). 

The advantage of $p$-adic cohomology theory is that it allows to determine Frobenius maps explicitly. For non-singular affine and projective varieties $X$ over $\Q$ for almost all primes $p$ the $p$-adic cohomology groups (of the reduction of $X$ modulo $p$) are isomorphic to $H^i_{dR}(X,\Q) \otimes_\Q \Q_p$, where $H^i_{dR}$ is the $i$th algebraic de Rham cohomology group. Roughly, the de Rham cohomology groups of an affine variety are given by 
\[
H^i_{dR}(X,\Q) = \frac{\text{ closed differential $i$-forms on $X$} }{\text{ exact differential $i$-forms on $X$}}
\]    
with $0 \le i \le \dim X$. Thus the Frobenius map is a $p$-adic operation on equivalence classes of differential forms which is responsible for counting points on $X$ over all $\F_{p^s}$. In \S\ref{sec:cartier-section} we will construct such an operation explicitly in the case of affine hypersurfaces. 

\section{Congruences}\label{sec:elementary-congs}

\subsection{Warming up with $p$-adic numbers}

\begin{lemma}\label{rasing-to-power-p-cong} Let $a, b \in \Z$ and $s \ge 1$. If $a \is b \mod {p^s}$ then $a^p \is b^p \mod {p^{s+1}}$. 
\end{lemma}
\begin{proof} Let us write $a = b + c p^s$ and raise to the $p$th power:
\[
a^p = (b + c p^s)^p = b^p + \sum_{i=1}^p \binom{p}{i} b^{p-i} c^{i} p^{si} \is b^p \mod {p^{s+1}}
\] 
because for $i=1$ we have $\binom{p}{1}p^s=p^{1+s}$ and terms with higher $i$ are divisible by $p^{is}$ with $is \ge 2s \ge s+1$. 
\end{proof}

For $a \in \Z$, let us start with the known congruence $a^p \is a \mod p$ and raise it to $p$th power multiple times. The above lemma yields
\[
a^{p^{s}} \is a^{p^{s-1}} \mod {p^s} \text{ for } s \ge 1.
\]  
It follows that the sequence $a^{p^s}$ has a limit in the ring of $p$-adic integers $\Z_p$. 

\begin{exercise} For $a \in \Z$ denote 
\[
\tau(a) = \text{$p$-adic }\lim_{s \to \infty} a^{p^s} \in \Z_p.
\]  Show that
\begin{itemize}
\item[(i)] $\tau(a) \is a \mod p$,
\item[(ii)] $\tau(a)^p = \tau(a)$, and
\item[(iii)] $\tau(a)=\tau(b)$ if $a \is b \mod p$.
\end{itemize}
\end{exercise}

It is clear that $\tau(a)=0$ when $p|a$ and otherwise $\tau(a) \in \Z_p^\times$ is a $p$-adic unit satisfying $\tau(a)^{p-1}=1$. Due to~(iii) the map $\tau$ is naturally defined on the set of residues $\Z/p\Z=\F_p$. We also conclude from the above exercise that there are $p-1$ different $(p-1)$st roots of unity in $\Z_p$ and they are naturally indexed by non-zero residues modulo $p$. We thus obtain a multiplicative character    
\[
\tau: \F_p^\times \to \Z_p^\times,
\]
which is called the \emph{Teichm\"uller character}. 

The existence of $(p-1)$st roots of unity as well as many other algebraic numbers in $\Z_p$ could be also derived from a fundamental lemma due to Kurt Hensel. 

\begin{exercise} \begin{itemize}
\item[(i)] Prove \emph{Hensel's lemma}: 

\noindent Let $P(T) \in \Z_p[T]$ be a monic polynomial and $\overline \alpha \in \F_p$ be a simple root of $P$ modulo $p$, that is $P(\overline \alpha) = 0$ and $P'(\overline \alpha) \ne 0$. Then there exists a unique lift $\alpha \in \Z_p$, $\alpha \is \overline \alpha \mod p$ such that $P(\alpha)=0$.    

\item[(ii)] Use Hensel's lemma to construct the Teichm\"uller character $\tau$.
\end{itemize}
\end{exercise}

We would like to observe that Lemma~\ref{rasing-to-power-p-cong} holds for elements $a,b$ of any unital commutative ring if one reads congruences modulo $p^s$ as modulo the principal ideal $(p^s)$ of this ring. The proof given above also reads in this more abstract setting. In the following subsection we will attempt to run the above process of multiple raising to power $p$ taking as the input a multivariate polynomial $f(\v x)$ with coefficients in $\Z$ instead of an integer $a \in \Z$. This situation is different already modulo $p$ because we have $f(\v x)^p \is f(\v x^p) \mod p$ and therefore one can not conclude that the sequence $f(\v x)^{p^s}$ tends to a limit.

\subsection{Lifting Hasse--Witt matrices}\label{sec:lifting-hasse-witt-matrices}
Our main object is a multivariable Laurent polynomial 
\[
f(\v x) \in R[x_1^{\pm 1}, \ldots, x_n^{\pm 1}]
\]
with coefficients in a commutative characteristic zero ring $R$. About this ring we will assume that 
\[
\cap_{s \ge 1} \; p^s R = \{ 0 \}.
\]
In particular, this defines the $p$-adic topology on $R$ (two elements are close when their difference belongs to $p^sR$ for a high $s$) and one can embed $R$ into its \emph{$p$-adic completion} 
\[
\hat{R} = \underset{\leftarrow}\lim \; R/p^sR.
\]   

The \emph{Newton polytope} of $f(\v x)$ is the convex hull in $\R^n$ of the set of exponent vectors of monomials occuring in $f(\v x)$. It is denoted by $\Delta  = \Delta(f)$. We shall use the notation $\v x^{\v u} = x_1^{u_1} \ldots x_n^{u_n}$ for $\v u \in \Z^n$. One can then write $f = \sum_{\v u \in \Z^n} f_{\v u} \v x^\v u$ with coefficients $a_\v u \in R$. The support ${\rm supp}(f) = \{ \v u: f_\v u \ne 0\}$ is a finite subset of $\Z^n$, and the Newton polytope $\Delta \subset \R^n$ is precisely the convex hull of ${\rm supp}(f)$.

Let $\Delta^\circ$ denote the interior of $\Delta$ in the Euclidean topology and
\[
\Delta^\circ_\Z = \Delta^\circ \cap \Z^n
\]
be the set of internal integral points of the Newton polytope.

\begin{multicols}{2}
\includegraphics{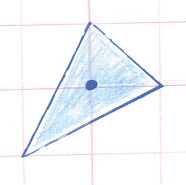}
\columnbreak

\begin{example} Here we see the Newton polytope $\Delta$ of the two variable 
 Laurent polynomial $f(\v x) = x_1 + x_2 + \frac1{x_1 x_2}$. Any other polynomial of the form 
  \[\bal
  f(\v x) = f_{(1,0)} \, x_1 + f_{(0,1)} \, x_2 \\
  + f_{(-1,-1)} \, \frac1{x_1 x_2} + f_{(0,0)}
  \eal\] with non-zero coefficients 
  \[f_{(1,0)},\; f_{(0,1)}, \; f_{(-1,-1)} \ne 0  \]
  will have same Newton polytope $\Delta$. The set of internal integral points of this polytope is $\Delta^\circ_{\Z} = \{ \v 0 \}$. 
 \end{example}
\end{multicols}

Assume that the number of internal integral points $g=\# \Delta^\circ_\Z$ is non-zero. We then define a sequence of $g \times g$ matrices $\{\beta_m; m \ge 1\}$ with entries in $R$ whose rows and columns are indexed by the elements of $\Delta^\circ_\Z$:  
\[
(\beta_m)_{\v u, \v v \in \Delta^\circ_\Z} = \text{ coefficient of } \v x^{\v v m - \v u} \text{ in } f(\v x)^{m-1}. 
\]
By convention, $\beta_1$ is the identity matrix. We are going to describe congruences satisfied by these matrices. 

A \emph{Frobenius endomorphism} $\sigma: R \to R$ is a ring endomorphism such that 
\[
\sigma(r) \is r^p \; \mod {pR} \quad \text{ for every } r \in R.
\]   
In other words $\sigma$ lifts the $p$th power endomorphism on the ring $R/pR$, and therefore it is also called a \emph{Frobenius lift}. For example, on $R=\Z$ the identity $\sigma=id$ is a Frobenius endomorphism. On the ring of polynomials $R=\Z[t]$ one can take $\sigma: r(t) \to r(t^p)$. 

In the following theorem Frobenius endomorphisms will be applied entry-wise to $g \times g$ matrices with entries in $R$.  

\begin{theorem}[\cite{HHW}]\label{HHW-main} For every Frobenius endomorphism $\sigma: R \to R$ one has
\[
\beta_{p^s} \is \beta_p \cdot \sigma(\beta_p) \cdot \ldots \cdot \sigma^{s-1}(\beta_p) \quad \mod {p }.
\]
Assuming that $\det(\beta_p)$ is invertible modulo $p$, one has congruenecs
\[
\beta_{p^{s+1}} \cdot \sigma(\beta_{p^s})^{-1} \is \beta_{p^{s}} \cdot \sigma(\beta_{p^{s-1}})^{-1} \quad \mod {p^s}
\]
for every $s \ge 1$.
\end{theorem}

Some remarks are needed to the statement of this theorem. In the first congruence $\mod p$ should be read as $\mod {p R^{g \times g}}$, where $R^{g \times g}$ is the ring of $g \times g$ matrices with entries in $R$. Next, we would like to note the following. 

\begin{remark}\label{p-adic-invertibility-rmk} An element $r \in R$ is invertible in the $p$-adic completion ring $\hat{R} \supseteq R$ if and only if its image in $R/pR$ is invertible. The implication $\Rightarrow$ is clear. For the inverse implication $\Leftarrow$ we observe that if $r,u, w \in R$ are such that $r u = 1 + p w$ then the multiplicative inverse of $r$ in $\hat{R}$ is given by $r^{-1} = u (1 + p w)^{-1}=u\sum_{m \ge 0}(-1)^m p^m w^m$. 
\end{remark}

If $\det(\beta_p)$ is invertible in $R/pR$ then all $\det(\beta_{p^s})$ are also invertible in $R/pR$ because the first congruence implies that $\det(\beta_{p^s}) \is \det(\beta_p)^{1+p+\ldots+p^{s-1}} \mod p$. Then $\beta_{p^s}$ and $\sigma(\beta_{p^s})$ are invertible as matrices with entries in $\hat{R}$, and the second congruence is valid modulo $\mod {p^s \hat{R}^{g \times g}}$. One can multiply it by the invertible element $\sigma(\det(\beta_{p^s})\det(\beta_{p^{s-1}}))$ to have a congruence modulo $p^s R^{g \times g}$.

The matrix $\beta_p$ is known as the \emph{Hasse--Witt matrix}. The reader can find a geometric interpretation of $\beta_p \mod p$ for non-singular projective hypersurfaces in~\cite{Katz-cong-zeta}. In the cited paper Nicholas Katz proves that the characteristic polynomial of $\beta_p$ is congruent modulo $p$ to a factor of the local zeta function over $\F_p$. If $\beta_p$ is invertible modulo $p$, which is the case whenever $\det(\beta_p)$ is an invertible element of $R/pR$, we say that the \emph{Hasse--Witt condition holds}, or that $p$ is an \emph{ordinary prime} for our hypersurface $f(\v x)=0$. A straightforward consequence of Theorem~\ref{HHW-main} is that when the Hasse--Witt condition holds then there exists the limiting matrix
\[
\Lambda_p = \text{ $p$-adic } \lim_{s \to \infty} \beta_{p^{s}} \cdot \sigma(\beta_{p^{s-1}})^{-1} \in \hat{R}^{g \times g}. 
\]  

\begin{exercise} Read the elementary proof of Theorem~\ref{HHW-main} in~\cite[\S 2-3]{HHW}. Unfortunately, that proof does not give any clues about the nature of matrices $\Lambda_p$.
\end{exercise}

\bigskip
In the case of $R=\Z$ we have associated to the Laurent polynomial $f(\v x) \in \Z[x_1^{\pm 1}, \ldots, x_n^{\pm 1}]$ a collection of $p$-adic matrices $\Lambda_p \in \Z_p^{g \times g}$ for all ordinary primes $p$.  Note that 
\[
\Lambda_p \is \beta_p \; \mod p
\]
and therefore the eigenvalues of $\Lambda_p$ are $p$-adic units. Numerical experiments in~\cite{HHW} suggested that, in the case $R=\Z$ the eigenvalues of $\Lambda_p$ are Frobenius roots of the hypersurface $f(\v x)=0$ over the finite field $\F_p$.    This claim was proved in~\cite[Thm A.1 and Rmk A.2]{DCI} and, under addtional geometric assumptions, also in~\cite{HLYY}. The authors of the latter paper used the technique from the paper ~\cite{Katz}, results of which were supporting the conjecture in~\cite{HHW} in the case ofhomogeneous polynomials defining a smooth projective hypersurface.  The fact that the eigenvalues of $\Lambda_p$ are Frobenius roots will be later demonstrated as a by-product of the theory developed in \S\ref{sec:cartier-section}-\ref{sec:periods}.  

In \S\ref{sec:cartier-section}-\ref{sec:periods} we will work with differential forms on the complement of the hypersurface $f(\v x)=0$. We will introduce a $p$-adic operation on them, the Cartier operation. Assuming the Hasse--Witt condition holds, we will construct a free $\hat{R}$-module of rank $g$ with a natural basis corresponding to internal integral points of the Newton polytope $\Delta$. We call this module the~\emph{unit-root crystal} associated to our hypersurface and a subset $\Delta^\circ \subset \Delta$. We will prove that $\Lambda_p$ is a matrix of the Cartier operation on the unit-root crystal (Theorem~\ref{DCI-Dwork-cong-thm} for $\mu = \Delta^\circ$). In \S\ref{sec:periods} this will give a conceptual proof of congruences in Theorem~\ref{HHW-main}, which is alternative to the above mentioned elementary proof given in~\cite{HHW}. 

In the rest of this section we will overview some examples in which matrices $\Lambda_p$ can be given explicitly. Their proofs will be postponed to later sections, when we will have all necessary tools in our hands.

\subsection{Atkin and Swinnerton-Dyer congruences} Consider 
\[
f(x,y)=y^2-x^3-A x-B \;\text{ with }\; A,B \in \Z
\]
such that the curve $f(x,y)=0$ is non-singular (elliptic curve, see Exercise~\ref{ell-curve-exercise}). Consider the expansion of the differential form $\omega=\frac12 dx/y$ in the local parameter $u=-x/y$ near the infinite point $O$:
\[
\frac{dx}{2y} = (1 + 2A u^4 + 3B u^6 + 6 A^2 u^8 + 20 BAu^{10}+\ldots)\, du = \left(\sum_{m \ge 1} \alpha_m u^{m} \right) \frac{du}{u}.
\] 
It turns out that $\alpha_m$ is the coefficient of $x^{m-1}$ in $(x^3 + Ax + B)^{(m-1)/2}$ when $m$ is odd and $\alpha_m=0$ when $m$ is even. In\cite[Example 0.4(a)]{Stienstra} Jan Stienstra attributes this observation to Frits Beukers.

\begin{exercise} Let $p$ be an odd prime not dividing the discriminant of $x^3+Ax+B$. Check that $\alpha_p \is a_p \mod p$, where $a_p = p-\#\{ (x,y)\in \F_p^2 : f(x,y)=0 \}$ is the Frobenius trace of this elliptic curve,~\eqref{ell-curve-zeta}.  
\end{exercise}

More generally, these expansion coefficients satisfy the Atkin and Swinnerton-Dyer congruences (\cite{ASD}):
\be{ASD-cong}
\alpha_m - a_p \,\alpha_{m/p} + p \, \alpha_{m/p^2} \is 0 \quad \mod {p^{\ord_p(m)}}.
\ee
Our $1 \times 1$ matrices $\beta_m$ are given by
\[
\beta_m = \text{ coefficient of } x^{m-1}y^{m-1} \text{ in } (y^2-x^3-A x-B)^{m-1} = (-1)^{(m-1)/2} \binom{m-1}{\frac{m-1}2} \alpha_m.
\]
Thus $\beta_p \is \alpha_p \is a_p \mod p$, and the Hasse--Witt condition holds for those primes $p$ for which this elliptic curve is ordinary. In this case Theorem~\ref{HHW-main} shows that $\beta_{p^s} \in \Z_p^\times$ for all $s \ge 1$. Since $\binom{p^s-1}{\frac{p^s-1}2} \in \Z_p^\times$, we conclude that $\alpha_{p^s} \in \Z_p^\times$. From the properties of the $p$-adic gamma function (\cite[\S IV.2]{Koblitz}) it is clear that
\[
\binom{p^s-1}{\frac{p^s-1}2} / \binom{p^{s-1}-1}{\frac{p^{s-1}-1}2} = \frac{\Gamma_p(p^s)}{\Gamma_p(\frac{p^s+1}2)^2} \overset{(p^s)}\is \frac{\Gamma_p(0)}{\Gamma_p(\frac{1}2)^2} = (-1)^{(p+1)/2} .
\] 
Multiplying by $(-1)^{\frac{p^s-p^{s-1}}{2}} = (-1)^{\frac{p-1}{2}}$ we obtain 1. Therefore 
\[
\beta_{p^{s}} / \beta_{p^{s-1}} \is \alpha_{p^{s}} / \alpha_{p^{s-1}} \quad  \mod {p^s}
\]
and $\Lambda_p = \lim_{s \to \infty} \beta_{p^{s}} / \beta_{p^{s-1}} = \lim_{s \to \infty} \alpha_{p^{s}} / \alpha_{p^{s-1}} $
Dividing the Atkin--Swinnerton-Dyer congruence $\alpha_{p^s} - a_p \alpha_{p^{s-1}} + p \alpha_{p^{s-2}} \is 0 \mod {p^s}$ by the $p$-adic unit $\alpha_{p^{s-2}}$ and tending $s \to \infty$ we conclude that the limit $\Lambda_p$ satisfies 
\[
\Lambda_p^2 - a_p \Lambda_p + p = 0.
\]
Hence $\Lambda_p \in \Z_p^\times$ is a Frobenius root of our ordinary elliptic curve and one has
\[
\# E(\F_{p^s}) = 1 + p^s - \Lambda_p^s - (p/\Lambda_p)^s, \quad s \ge 1.
\] 
The fact that $\Lambda_p$ is a Frobenius root will also follow from our constructions, see Corollary~\ref{eigenvalues-of-cartier-cor}, and Atkin and Swinnerton--Dyer congruences can be back-engineered from this fact. 

In~\cite{Stienstra} Stienstra generalizes the Atkin and Swinnerton--Dyer congruences congruences to a wide class of projective varieties.    

\subsection{Dwork congruences}\label{sec:dwork-cong}
Take 
\[
f(\v x)=1-t g(\v x)
\]
with $g(\v x) \in \Z[x_1^{\pm 1}, \ldots, x_n^{\pm 1}]$. Assume that the Newton polytope $\Delta$ of $g(\v x)$ has only one interior integral point, and that this point is the origin: $\Delta_\Z^\circ = \Delta^\circ \cap  \Z^n = \{\v 0 \}$. In this case $\Delta$ is also the Newton polytope of $f(\v x)$. One can start with $R=\Z[t]$. 

For $i \ge 0$ we denote by $c_i$ the constant term of $g(\v x)^i$. Consider the formal series $\gamma(t)=\sum_{i=0}^\infty c_i t^i \in \Z[[t]]$ and its truncations $\gamma_m(t) = \sum_{i=0}^{m-1} c_i t^i$. We have
\[
\beta_p(t) = \text{ constant term in } (1 - t g(\v x))^{p-1} = \sum_{i=0}^{p-1} (-1)^i \binom{p-1}{i} c_i t^i \is \gamma_p(t) \; \mod p.
\]
The Hasse-Witt condition will be satisfied if we add the inverse of the polynomial $\beta_p(t)$ to our ring, so we take the bigger ring $R=\Z[t, \beta_p(t)^{-1}] \subset \Z\lb t\rb$. Its $p$-adic completion $\hat{R} \subset  \Z[t,\beta_p(t)^{-1}]\,\hat{\;} \subset \Z_p\lb t\rb$ consists of series that can be approximated $p$-adically by rational functions whose denominators are powers of $\beta_p(t)$. As a consequence if Theorem~\ref{HHW-main} there exists the $p$-adic limit
\[
\Lambda_p(t) = \text{$p$-adic }\; \lim_{s \to \infty} \frac{\beta_{p^s}(t)}{\beta_{p^{s-1}}(t^p)} \in  \Z[t,\beta_p(t)^{-1}]\,\hat{\;}.
\]
The computation of this $p$-adic limit can be achieved by careful analysis of congruence properties for the binomial coefficients $(-1)^i\binom{p^s-1}{i}$ occuring in $\beta_{p^s}(t)$ in contrast to $\gamma_{p^s}(t)$, and one finds that in fact
\[
\Lambda_p(t) = \frac{\gamma(t)}{\gamma(t^p)}.
\]
The reader may find a similar limit computation in~\cite[Example 5.5]{DCI}. We prefer to omit the details here, as they will distract us from our main topic. Moreover, in the proof of Theorem~\ref{frob-unit-root-p-adic-analytic-thm} we will be able to obtain this expression for $\Lambda_p(t)$ in a more conceptual way. One can conclude that
\be{Dwork-unit-root-general}
\frac{\gamma(t)}{\gamma(t^p)} \in \Z[t,\gamma_p(t)^{-1}]\,\hat{\;},
\ee
which follows from our expression for $\Lambda_p(t)$ and the fact that the $p$-adic completions $p$-adic completions $\Z[t,\beta_p(t)^{-1}]\,\hat{\;}$ and $\Z[t,\gamma_p(t)^{-1}]\,\hat{\;}$ are equal (see the  exercise below).

\begin{exercise}\label{frob-lift-exercise} 
Let $s(t) \in \Z[t]$ be a polynomial such that $p \nmid s(0)$. Consider the $p$-adic completion $\Z[t,s(t)^{-1}]\,\hat{\;}$:

\begin{itemize}
\item[(i)]embed this ring into $\Z_p\lb t \rb$;
\item[(ii)] show that this ring depends only on $s(t) \mod p$, that is for any $\tilde s(t)\in \Z[t]$ such that $\tilde s(t) \is s(t) \mod p$ one has $\Z[t,s(t)^{-1}]\,\hat{\;} = \Z[t,\tilde s(t)^{-1}]\,\hat{\;}$;
\item[(iii)] show that the Frobenius lift $\sigma: t \mapsto t^p$ on $\Z_p\lb t \rb$ preserves the ring $\Z[t,s(t)^{-1}]\,\hat{\;}$.
\end{itemize}
\end{exercise}

\bigskip

Let us look again at~\eqref{Dwork-unit-root-general}. It is a very non-trivial conclusion that the series $\gamma(t)/\gamma(t^p)$ can be approximated $p$-adically by rational functions. This fact and particular examples of such approximations were discovered by Bernard Dwork for certain class of hypergeometric series $\gamma(t)$, see e.g.~\cite[\S 5]{Dw-deformation} and~\cite[\S 3]{Dw-p-adic-cycles}. His results can be generalized in our setup in the following way.  

\begin{theorem}[Dwork's congruences,~\cite{MV-16} and~\cite{DCII}]\label{dwork-congs-them}  Let $g(\v x) \in \Z[x_1^{\pm 1}, \ldots, x_n^{\pm 1}]$ be a Laurent polynomial such that the origin $\v 0 \in \R^n$ is the only interior integral point in its Newton polytope. Consider the generating series of constant terms of its powers
\[
\gamma(t)=\sum_{i=0}^\infty c_i t^i, \qquad c_i = \text{ coefficient of $\v x^\v 0$ in } g(\v x)^i.
\]
Then for any prime $p$ and integer $m \ge 1$ we have
\[
\frac{\gamma(t)}{\gamma(t^p)} \is \frac{\gamma_{m}(t)}{\gamma_{m/p}(t^p)} \qquad \mod {p^{\ord_p(m)}},
\]        
where  $\gamma_m(t) = \sum_{i=0}^{m-1} c_i t^i$ denotes the truncation of the beginning $m$ terms in $\gamma(t)$.
\end{theorem}

The proof of this theorem will become apparent in \S\ref{sec:periods-mod-p-s} (see~Exercise~\ref{Dwork-cong-exercise}). Dwork noted that~\eqref{Dwork-unit-root-general} yields a way to evaluate $\gamma(t)/\gamma(t^p)$ at certain points $t_0 \in \Z_p^\times$ while the series may not be convergent on the $p$-adic unit circle. Namely, when $\gamma_p(t_0) \ne 0 \mod p$ one can evaluate
\[
\Lambda_p(t_0) = \frac{\gamma(t)}{\gamma(t^p)} \Big|_{t=t_0} = \text{ $p$-adic} \lim  \frac{\gamma_{p^s}(t_0)}{\gamma_{p^{s-1}}(t_0^p)}.
\]  
Suppose that $t_0^p = t_0$ or, equivalently, $t_0 = \tau(a)$ is the Teichm\"uller lift of $a \in \F_p$. Then the above evaluation is compatible with the computation of the limiting value $\Lambda_p$ for the hypersurface $1-t_0 g(\v x)=0$. We conclude that whenever $\gamma_p(a) \ne 0$ then $\Lambda_p(\tau(a))$ is a Frobenius root of the hypersurface $1-a g(\v x)=0$ over $\F_p$. This fact demonstrates the relation between period functions, like $\gamma(t)$, and local zeta functions of fibres of this family. Initially observed by Dwork in~\cite{Dw-deformation}, this phenomenon led to understanding of how $p$-adic Frobenius matrices behave in families. We will return to this topic in \S\ref{sec:frob-structure}.  

\subsection{Formal expansions of rational functions}\label{sec:formal-exp}

Let $f(\v x) = \sum f_\v u \v x^\v u \in R[x_1^{\pm 1},\ldots,x_n^{\pm 1}]$ and $\Delta \subset \R^n$ be its Newton polytope. There are different ways to expand rational functions $h(\v x)/f(\v x)^m$ with $h(\v x) \in R[x_1^{\pm 1},\ldots,x_n^{\pm 1}]$ and $m \ge 1$ into formal Laurent series. Such expansions will be useful for us in the constructions in the following sections.

Let us describe an expansion procedure with respect to vertices of $\Delta$. Pick a vertex $\v b \in \Delta$ and assume that $f_\v b \in R^{\times}$. Then

\[\bal
\frac{h(\v x)}{f(\v x)^m} &= \frac{h(\v x)}{f_{\v b}^m \v x^{m \v b}(1 + \ell(\v x))^m} = \frac{h(\v x) \v x^{-m \v b}}{f_{\v b}^m} \sum_{s \ge 0} \binom{-m}{s} \ell(\v x)^s = \sum_{\v v \in \Z^n} c_{\v v} \v x^{\v v}.
\eal
\]
Here $\binom{-m}{s} = \frac{-m \cdot (-m-1) \cdot \ldots \cdot (-m-(s-1))}{s!} = (-1)^s \binom{s+m-1}{m-1}$ are integers. 

\begin{multicols}{2}
\includegraphics[width=4cm]{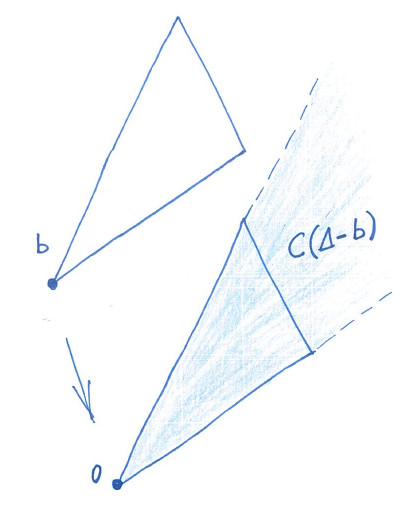}
\columnbreak

\noindent In this computation $\ell(\v x)$ is the Laurent polynomial $\ell(\v x) = f_{\v b}^{-1} \v x^{-\v b} f(\v x) - 1$ supported in the cone $C(\Delta - \v b)$. Since $\ell(\v x)$ has no constant term, only finitely many summands $\ell(\v x)^s$ will contribute to each monomial $\v x^\v v$, and hence the above series is well-defined.

For later we would like to remark that if $supp(h) \subset m \Delta$ then the formal expansion $\sum_{\v v \in \Z^n} c_{\v v} \v x^{\v v}$ is itself supported in the cone $C(\Delta - \v b)$.

\end{multicols}

Examples of such formal expansions will be given in the next \S.

\subsection{Gauss congruences}\label{sec:gauss-cong}
Let $f(\v x), h(\v x) \in \Z[x_1^{\pm 1},\ldots,x_n^{\pm 1}]$, $\Delta$ is the Newton polytope of $f(\v x)$ and $\supp(h) \subseteq \Delta$. We expand the ratio of these functions at one of the vertices of $\Delta$ as was explained in \S\ref{sec:formal-exp}:
\be{gauss-cong-expansion}
\frac{h(\v x)}{f(\v x)} = \sum_{\v v \in \Z^n} c_\v v \v x^\v v.
\ee

\begin{theorem}{(Gauss' congruences,~\cite{BHS18} and~\cite{DCI})}\label{gauss-thm} If the only integral points in $\Delta$ are its vertices then for every odd prime $p$ which does not divide any of the coefficients of $f(\v x)$ the coefficients in the expansion~\eqref{gauss-cong-expansion} satisfy
\[
c_{\v v} \is c_{\v v/p} \; \mod {p^{\ord_p(\v v)}}, \qquad \forall \v v \in \Z^n.
\]
\end{theorem}

\bigskip

We will prove this fact in \S\ref{sec:gauss-cong-proof}. These congruences are named after Gauss due to the basic one-variable example 
\[
f(x)=\frac1{1-ax}=\sum_{m=0}^\infty a^m x^m, \qquad a \in \Z.
\]
In this case for every prime $p \nmid a$ one has $a^m \is a^{m/p} \mod {p^{\ord_p(m)}}$. The reader may notice that this simple Gauss' congruence follows immediately from Lemma~\ref{rasing-to-power-p-cong}.

As our second example, let us expand at $\v 0$ the rational function
\[
\frac1{1-x_1-x_2} = \sum_{m \ge 0}(x_1+x_2)^m = \sum_{\v v \in \Z_{\ge 0}^2} \binom{v_1+v_2}{v_1} \v x^\v v.
\]
We conclude from Theorem~\ref{gauss-thm} that for any odd prime $p$ one has
\be{gauss-binomial-weak}
\binom{v_1+v_2}{v_1} \is \binom{(v_1+v_2)/p}{v_1/p} \mod {p^{\min(\ord_p{v_1},\ord_p{v_2})}}.
\ee
In fact this rational function is very special: though the modulus $p^{\ord_p(\v v)}$ in the Gauss congruence is sharp generically, in the case ~\eqref{gauss-binomial-weak} one actually has congruences modulo twice higher power of $p$, that is $p^{2 \min(\ord_p{v_1},\ord_p{v_2})}$. This case is an example of a \emph{supercongruence} which we will be able to explain only in~\S\ref{sec:exc-lifts}.

\section{$p$-adic Cartier operation on differential forms}\label{sec:cartier-section}

Let us recall the setup of \S\ref{sec:lifting-hasse-witt-matrices}. We work with a Laurent polynomial 
$f(\v x) = \sum_{\v u} f_{\v u} \v x^{\v u} \in R[x_1^{\pm 1},\ldots,x_n^{\pm 1}]$. The coefficients $f_{\v u}$ belong to a characteristic zero ring $R$. The Newton polytope of $f(\v x)$ is denoted by $\Delta \subset \R^n$. This polytope is the convex hull of $\supp(f)=\{ \v u \in \Z^n: f_{\v u} \ne 0 \}$. 

\subsection{Differential forms on the complement of the zero locus of $f(\v x)$}\label{sec:diff-forms}

Algebraic differential n-forms on the complement of the hypersurface $X_f = \{ \v x: f(\v x) = 0 \}$ in the torus $\mathbb{T}^n = \{ \v x :   x_i \ne 0 \; \forall i \}$ are of the shape $\omega = \frac{h(\v x)}{f(\v x)^m} \frac{d x_1}{x_1}\ldots \frac{d x_n}{x_n}$ with $m \ge 1$ and a Laurent polynomial $h \in R[x_1^{\pm 1},\ldots,x_n^{\pm 1}]$. We will now define submodules of our interest in the $R$-module of differential forms.

A subset $\mu \subseteq \Delta$ will be called \emph{open} if its complement $\Delta \setminus \mu$ is a union of faces of $\Delta$ of any dimensions. Below are examples of open subsets and their sets of integral points which we denote by $\mu_{\Z} = \mu \cap \Z^n$: 

\includegraphics[width=12cm]{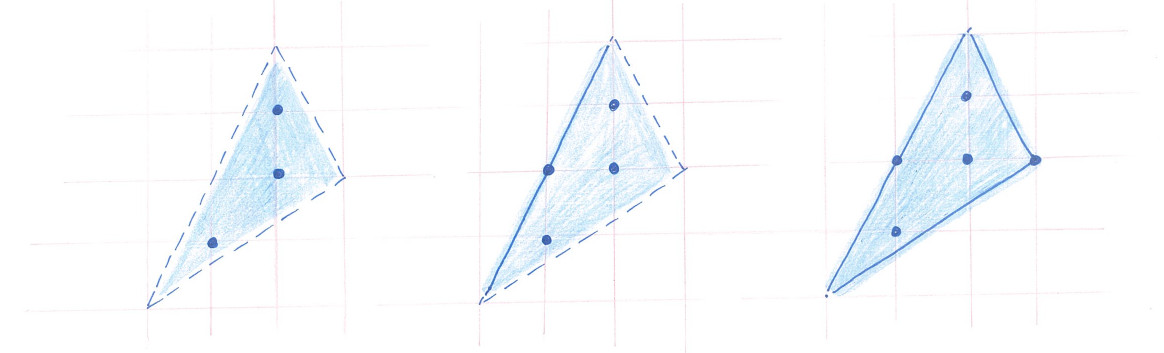} 

We define the $R$-modules
\[
\Omega_f(\mu)  = \left\{ \text{\begin{small}$(m-1)!$\end{small}}\dfrac{ h(\v x)}{f(\v x)^m}\;\;\Big|\;\;\bal  &m \ge 1, \quad h \in R[x_1^{\pm 1},\ldots,x_n^{\pm 1}]\quad\\
 & supp(h) \subset m \mu \eal \right\}. \\
\]
The module $\Omega_f(\Delta)$ is also denoted by $\Omega_f$. We shall also consider
\[ d \Omega_f =  R\text{-module generated by } x_i \frac{\partial \nu}{\partial x_i}, \; \nu \in \Omega_f, \; i=1,\ldots, n.
\]
One can easily check that derivations $x_i \frac{\partial }{\partial x_i}$ preserve modules $\Omega_f(\mu)$. In particular, it follows that $d \Omega_f = \sum_{i=1}^n x_i \frac{\partial }{\partial x_i}(\Omega_f)$ is a submodule in $\Omega_f$. 

\begin{remark}\label{geometric-remark} In~\cite{Batyrev} Batyrev gives the following identification when $R=\C$ and $f$ is \emph{$\Delta$-regular}:
\[\bal \Omega_f / d \Omega_f &\cong H_{dR}^n(\mathbb{T}^n \setminus X_f)\\
 \Omega_f \ni \dfrac{ h(\v x)}{f(\v x)^m} &\mapsto  \dfrac{ h(\v x)}{f(\v x)^m} \frac{d x_1}{x_1} \ldots \frac{d x_n}{x_n}  \\
 d \Omega_f & \leftrightarrow \text{ exact forms } 
 \eal\]
 De Rham cohomology of affine algebraic varieties over $\C$ has a mixed Hodge structure, which consists of two filtrations. Batyrev also proves that uder the above identification the filtration  
\[
\Omega_f^{(1)} \subset \Omega_f^{(2)} \subset \ldots  \quad\text{ with }\quad \Omega_f^{(k)} = \left\{ \frac{h(\v x)}{f(\v x)^m} \in \Omega_f \,|\, m \le k \right\}
\]
descends to the \emph{Hodge filtration}, and the \emph{weight filtration} on de Rham cohomology is the image of
\[
\Omega_f(\mu_n) \supset \Omega_f(\mu_{n-1}) \supset \ldots \supset \Omega_f(\mu_{1}),
\]
where $\mu_\ell$ is the complement of the union of all faces of codimension  $\ge \ell$.
\end{remark}

In what follows we will work with elements of $\Omega_f$, which are rational functions. However it is useful to remember that they correspond to differential forms.    

\subsection{The Cartier operation}\label{sec:cartier}
Let us now fix a prime $p>2$ and assume that $\cap_{s \ge 1} p^s R = \{ 0 \}$. Pick a vertex $\v b \in \Delta$ and consider the following operation on formal expansions with respect to $\v b$:
\[
\cartier: \frac{h(\v x)}{f(\v x)^m}  = \sum_{\v v }  c_\v v \v x^\v v  \mapsto \sum_{\v v } c_{p \v v} \v x^\v v .
\]
It is unlikely that the resulting series is again an expansion of a rational function. However it can be approximated by rational functions $p$-adically. Recall that a Frobenius lift $\sigma: R \to R$ is a ring endomorphism such that $\sigma(r) - r^p \in p R$ for all $r \in R$. For an endomorphism $\sigma$ and $f(\v x)= \sum_\v u f_\v u \v x^\v u$ we denote $f^\sigma(\v x) = \sum_\v u \sigma(f_\v u) \v x^\v u$.

\begin{lemma} Let $\sigma: R \to R$ be a Frobenius lift. For $\frac{h(\v x)}{f(\v x)^m} = \sum c_\v v \v x^\v v$,  the series $\sum c_{p\v v} \v x^\v v  $ can be approximated $p$-adically by rational functions with powers of $f^\sigma(\v x)$ in the denominator. More precisely, 
\[
\cartier (\Omega_f) \;\subset \; \widehat \Omega_{f^\sigma} = \text{ $p$-adic completion of } \Omega_{f^\sigma}.
\]
\end{lemma}
\begin{proof} We reproduce the proof of \cite[Prop. 3.3]{DCI}. It suffices to consider $h(\v x)=(m-1)! \v x^\v u$ with $\v u \in m\Delta$ and show that $\cartier((m-1)! \v x^\v u /f(\v x)^m) \in \hat\Omega_{f^\sigma}$.  We rewrite $1/f(\v x)^m$ as 
$f(\v x)^{p\lceil m/p\rceil-m}/f(\v x)^{p\lceil m/p\rceil}$.
Then note that
$f(\v x)^p=f^\sigma(\v x^p)-pG(\v x)$ for some Laurent polynomial $G$
with coefficients in $R$ and support in $p\Delta$. Then we use the $p$-adically convergent expansion
\[
\frac{\v x^{\v u}}{f(\v x)^{m}}=
\frac{\v x^{\v u}f(\v x)^{p\lceil m/p\rceil-m}}
{(f^\sigma(\v x^p)-pG(\v x))^{\lceil m/p\rceil}}
=\sum_{r\ge0}p^r{\lceil m/p\rceil+r-1\choose r}
\frac{G(\v x)^r}{f^\sigma(\v x^p)^{r+\lceil m/p\rceil}}
\v x^{\v u}f(\v x)^{p\lceil m/p\rceil-m},
\]
multiply it with $(m-1)!$ and apply $\cartier$. We find that
\be{cartier-formula}
\cartier\left((m-1)! \frac{\v x^\v u }{f(\v x)^m}\right)=\sum_{r\ge0}\frac{p^r}{r!}
\frac{(m-1)!}{(\lceil m/p\rceil-1)!}
(\lceil m/p\rceil+r-1)!\frac{Q_r(\v x)}{f^\sigma(\v x)^{r+\lceil m/p\rceil}},
\ee
where the $Q_r(\v x)=\cartier(G(\v x)^r\v x^{\v u}f(\v x)^{p\lceil m/p\rceil-m})$ are
Laurent polynomials in $x_1,\ldots,x_n$ with support in
$(\lceil m/p\rceil+r)\Delta$ and coefficients in $R$. Since the $p$-adic valuations of $p^r/r!$ grow infinitely, the right-hand sum is a $p$-adically convergent sum of elements of $\Omega_{f^\sigma}$, and hence it belongs to $\widehat\Omega_{f^\sigma}$.
\end{proof}

\if
 The last formula can be rewritten as
\begin{equation}\label{cartiermatrix}
\cartier(\omega_{\v u})=\sum_{\v v\in C(\Delta)_\Z^+}F_{\v u,\v v}
\omega^\sigma_{\v v},
\end{equation}
where 
\be{cartiermatrix-entries}
F_{\v u,\v v} = \bcs \frac{p^r}{r!} \frac{(m-1)!}{(\lceil m/p\rceil-1)!} \times \text{ coefficient of } \v x^{\v v} \text{ in } Q_r(\v x), & r:=v_0-\lceil u_0/p\rceil\ge0, \\
0 , & v_0<\lceil u_0/p\rceil. \ecs
\ee
To show that $\cartier(\omega_{\v u})\in\hat\Omega_{f^\sigma}$
it suffices to show that
$\ord_p(F_{\v u,\v v})\to\infty$ as $v_0\to\infty$. To that end we observe that
\[
\ord_p(F_{\v u,\v v})\ge r-\ord_p(r!)+\ord_p\left(\frac{(u_0-1)!}
{(\lceil u_0/p\rceil-1)!}\right).
\]
It is straightforward to see that $\frac{(u_0-1)!}{(\lceil u_0/p\rceil-1)!}$ has
order $\ge \lceil u_0/p\rceil-1$ and that $\ord_p(r!)<\frac{r}{p-1}$. This gives us
\be{cartier-entries-estimate}
\ord_p(F_{\v u,\v v})\ge r+\lceil u_0/p\rceil-1-\frac{r}{p-1}
=v_0-1-\frac{r}{p-1}\ge\frac{p-2}{p-1}(v_0-1).
\ee
The latter goes to $\infty$ with $v_0$ when $p>2$.
\fi

From now on it is convenient to {\bf assume that $R$ is $p$-adically complete}. The  $R$-linear operation
\
\be{cartier-def}
\cartier: \widehat\Omega_f \to \widehat\Omega_{f^\sigma}, \quad \sum  c_\v v \v x^\v v  \mapsto \sum  c_{p \v v} \v x^\v v 
\ee 
is called the \emph{$p$-adic Cartier operation}.  We shall now list its properties.

\begin{theorem}\label{cartier-props}
\begin{itemize}
\item[(i)] The operation $\cartier: \widehat\Omega_f \to \widehat\Omega_{f^\sigma}$ is independent of the choice of vertex $\v b \in \Delta$ with respect to which the formal expansion in~\eqref{cartier-def} is done.
\item[(ii)] For any open set $\mu \subseteq \Delta$ the image of $\widehat\Omega_f(\mu)$ lies in $\widehat\Omega_{f^\sigma}(\mu)$.  
\item[(iii)] For $1 \le i \le n$ we have
\[
\cartier \circ x_i \frac{\partial}{\partial x_i} = p \, x_i \frac{\partial}{\partial x_i} \circ \cartier.
\]
In particular, $\cartier(d \widehat \Omega_f) \subset d \widehat \Omega_{f^\sigma}$ and the Cartier operation descends to the homological quotients
\[
\cartier: \widehat \Omega_f / d \widehat \Omega_f \to \widehat \Omega_{f^\sigma} / d \widehat \Omega_{f^\sigma}. 
\]

\item[(iv)] If $R=\Z_p$ and $\sigma=id$ then
\[
(p^s-1)\; {\rm Trace}(\,\cartier^s \,|\, \hat \Omega_f\,) = \# ( \T^n \setminus X_f)  (\F_p^s)
\]
for every $s \ge 1$.
\end{itemize}
\end{theorem}

\begin{proof}[Sketch of proof]
Part~(i) follows from formula~\eqref{cartier-formula} which actually describes the Cartier map without appealing to the formal expansion procedure. Part~(ii) can be also deduced from this formula, because our definition of finite topology in \S\ref{sec:diff-forms} implies that when $\v u in m \mu$ then $\supp (\Q_r) \subset (\lceil m/p \rceil + r)\mu$. The reader can find a detailed explanation of this in~\cite[Prop 3.4]{DCI}. Part~(iii) can be easily checked on formal expansions. 

Part~(iv) is proved in~\cite[Thm A.1]{DCI}. There are obvious relations among the generators $(m-1)!\frac{\v x^\v u}{f(\v x)^m}$ of the module $\Omega_f$. We consider its resolution in which such generators become free and lift the Cartier action to such free modules in~\cite[Thm A.5]{DCI}. These free modules with the Cartier action turn out to be the \emph{exponential modules} introduced by Dwork, and the desired computation of the trace    on them is precisely the \emph{Dwork trace formula} (\cite[Prop A.8]{DCI}). We remark that the traces here are well defined because modulo any power of $p$ the image of $\cartier$ is of finite rank, and the same property holds true for the lift of this operation to the exponential module.   
\end{proof}

\subsection{Formally exact forms}\label{sec:formally-exact} In \S\ref{sec:cartier} we defined the Cartier operation thorough its action on formal expansions with respect to a vertex $\v b \in \Delta$. Recall that the procedure of formal expansion was defined in \S\ref{sec:formal-exp}. If the coefficient in $f(\v x)$ at $\v x^\v b$ is in $R^\times$, this operation embeds $\Omega_f$ into the $R$-module 
\[
\Omega_{\rm formal} = \{ \sum_{\v u \in C(\Delta - \v b)\cap \Z^n} c_\v u \v x^\v u \;|\; c_\v u \in R \}
\]
of formal Laurent series supported in the  cone $C(\Delta-\v b)$. We define the submodule of \emph{formal derivatives} as
\[
d \Omega_{\rm formal} = \sum_{i=1}^n  x_i \frac{\partial}{\partial x_i} (\Omega_{\rm formal}).
\]

\begin{lemma}\label{Katz-lemma} A series $\nu=\sum_{\v u} c_\v u \v x^\v u \in \Omega_{\rm formal}$ is a \emph{formal derivative} if and only if one of the following equivalent conditions holds
\begin{itemize}
\item[(i)] $c_\v u \in g.c.d.(u_1,\ldots,u_n) R$ for each $\v u$,
\item[(ii)] $\cartier^s(\nu) \in p^s \Omega_{\rm formal}$ for each $s \ge 1$.
\end{itemize}
\end{lemma}

Here criterion~(i) holds in general over rings of characteristic~0, while for~(ii) to be true one needs $R$ to be a $\Z_p$-algebra, so that all other primes are invertible in $R$. Recall that we assume $R$ to be $p$-adically complete, and hence it is a $\Z_p$-algebra.

\begin{proof}
We leave this proof as an exercise.
\end{proof}

For a subset $\mu \subseteq \Delta$ which is \emph{open} in the finite topology of \S\ref{sec:diff-forms}, we define the sequence of square matrices with entries in $R$ indexed by the set $\mu_\Z = \mu \cap \Z^n$ as 
\[
\beta_m(\mu)_{\v u, \v v \in \mu_\Z} = \text{ coefficient of } \v x^{m \v v - \v u} \text { in } f(\v x)^{m-1}.
\]
The matrix $\beta_p(\mu)$ will be called the Hasse-Witt matrix corresponding to $\mu$ and denoted by 
\[
HW(\mu)=\beta_p(\mu).
\]
We shall now state the key result of our paper \emph{Dwork crystals I}. We say that \emph{the Hasse--Witt condition holds for $\mu$} if $\det(HW(\mu))$ is invertible modulo $p$. Under our assumption that $R$ is $p$-adically complete this is equivalent to being invertible in $R$, see Remark~\ref{p-adic-invertibility-rmk}.

\begin{theorem}\label{DCI-main}(\cite[Thm 4.3]{DCI}) Assume that $\cap_{s \ge 1} p^s R = \{0\}$ and $R$ is $p$-adically complete. If the Hasse--Witt condition holds for an \emph{open} set $\mu \subseteq \Delta$, then we have the following direct sum decomposition of $R$-modules 
\[
\Omega_f(\mu) = \Omega_f^{(1)}(\mu) \oplus \left(d \Omega_{\rm formal} \cap \Omega_f(\mu) \right),
\]
where 
\[
\Omega_f^{(1)}(\mu) = \sum_{\v u \in \mu_\Z} R \, \frac{\v x^\v u}{f(\v x)}
\]
is a free module of rank $\# \mu_\Z$. Moreover, we have
\[
\cartier (\Omega_f(\mu)) \subset \Omega_{f^\sigma}^{(1)}(\mu) + p \, d \Omega_{\rm formal}.
\]
\end{theorem}

It is clear from Lemma~\ref{Katz-lemma} that the Cartier operation preserves $d\Omega_{\rm formal}$ (and even maps it to $p \, d\Omega_{\rm formal}$). Hence, under the assumptions of Theorem~\ref{DCI-main}, one can define the matrix of the Cartier operation on the quotient by formal derivatives $\Omega_f(\mu) / ( d \Omega_{\rm formal} \cap \Omega_f(\mu))$. We call this quotient the \emph{unit-root crystal}. There is a unique matrix $\Lambda(\mu) = (\lambda_{\v u, \v v})_{\v u, \v v \in \mu_\Z}$ with entries in $R$   
such that
\be{cartier-matrix-def-1}
\cartier\left(\frac{\v x^\v u}{f(\v x)}\right) = \sum_{\v v \in \mu_\Z} \lambda_{\v u, \v v} \frac{\v x^\v v}{f^\sigma(\v x)} \quad \mod {p \, d \Omega_{\rm formal}}
\ee
for every $\v u \in \mu_\Z$. In Section~\ref{sec:periods} we are going to focuse on the computation of $\Lambda(\mu)$. We remark that this matrix depends on the chosen Frobenius lift $\sigma$.

The proof of Theorem~\ref{DCI-main} will be given in \S\ref{sec:contraction}. It will exploit the $p$-adic contraction property of the Cartier operator stated in the following proposition.  

\begin{proposition}\label{Cartier-mod-p} One has 
\[
\cartier (\Omega_f(\mu)) \subset \Omega_{f^\sigma}^{(1)}(\mu) + p \, \widehat\Omega_{f^\sigma}(\mu).
\]
 For any $\v u \in \mu_\Z$ we have
\[
\cartier\left(\frac{\v x^\v u}{f(\v x)}\right) = \sum_{\v v \in \mu_\Z} HW_{\v u,\v v} \frac{\v x^\v v}{f^\sigma(\v x)} \quad \mod {p \widehat\Omega_{f^\sigma}(\mu)}. 
\]
\end{proposition}
\begin{proof} We analyze formula~\eqref{cartier-formula}. $p^r/r!$ is divisible by $p$ for all $r>0$ and $(m-1)!/(\lceil m/p \rceil-1)!$ is divisible by $p$ for $m > p$. Terms with $r=0$ and $m \le p$ have in their denominators $f^\sigma(\v x)$ in the power $r + \lceil m/p \rceil=1$, which proves our first claim. For the second claim we take $r=0,m=1$ and obtain that $Q_0(\v x)=\cartier\left( \v x^\v u f(\v x)^{p-1}\right)$, which proves the claimed formula for the Cartier action modulo $p$ with summation over $\v v \in \Delta_\Z$. The fact that $HW_{\v u , \v v} = 0 \mod p$ when $\v v \notin \mu$ follows from~(ii) in Theorem~\ref{cartier-props}.  
\end{proof}

Note that this proposition shows that under the conditions of Theorem~\ref{DCI-main} one has 
\[
\Lambda(\mu) \is HW(\mu) \mod p,
\]
where $\Lambda(\mu)$ is the matrix of Cartier action modulo formal derivatives defined in~\eqref{cartier-matrix-def-1}.

\begin{corollary}\label{eigenvalues-of-cartier-cor} Suppose that $R=\Z_p$ and the Hasse--Witt condition holds for the whole Newton polytope $\Delta$. Then the eigenvalues of the Cartier matrix $\Lambda=\Lambda(\Delta)$ defined by the condition~\eqref{cartier-matrix-def-1} with $\mu=\Delta$ are  Frobenius roots of the toric hypersurface $X_f = \{ \v x \in \T^n | f(\v x) = 0\}$ over the field $\F_p$. Moreover, they are all Frobenius roots of $p$-adic valuation less than~1. 
\end{corollary}
\begin{proof} We reproduce the argument from~\cite[Remark A.2]{DCI}. By part~(iv) of Theorem~\ref{cartier-props} we have for all $s \ge 1$
\[
(p^s-1)\; {\rm Trace}(\,\cartier^s \,|\, \hat \Omega_f\,) = \# ( \T^n \setminus X_f)  (\F_p^s) = (p^s-1)^n - \#X_f(\F_{p^s}).
\]
Since $\cartier$ is divisible by $p$ on formal derivatives, this yields
\[
{\rm Trace}(\Lambda^s) \is {\rm Trace}(\,\cartier^s \,|\, \hat \Omega_f\,) \is  1 + (-1)^{n+1} \#X_f(\F_{p^s}) \mod {p^s}.
\]
Our claim follows from this.
\end{proof}

\begin{problem} There are examples when the Hasse--Witt condition holds for a subset $\mu \subsetneq \Delta$ but not for the whole $\Delta$. We expect that in this case the eigenvalues of $\Lambda(\mu)$ will be Frobenius roots as well, however the proof of~(iv) in Theorem~\ref{cartier-props} was given in the appendix to~\cite{DCI} only in the case of $\mu=\Delta$. Can one adopt the Dwork trace formula to subsets $\mu$ open in the topology of \S\ref{sec:diff-forms}? 
\end{problem}

If the above question is answered positively, one would obtain a combinatorial structure on the Frobenius roots in the form of their belonging to open sets $\mu \subseteq \Delta$. This is particularly interesting in the view of Remark~\ref{geometric-remark}, where the weight filtration on de Rham cohomology is given in terms of a sequence of open subsets of $\Delta$. 

\subsection{The contraction property}\label{sec:contraction}

\begin{proposition}\label{key-contraction-lemma}
Let $M_0,M_1,M_2,\ldots$ be an infinite sequence of $R$-modules and
$\phi_i:M_{i-1}\to M_i$ $R$-linear maps for all $i\ge1$.
Suppose that $\cap_{s\ge1}p^sM_i=\{0\}$ for all $i$. For each $i$ let $N_i$ be a 
submodule of $M_i$ such that $\phi_i(M_{i-1})\subset N_i+pM_i$ for all $i\ge1$.
Suppose that $N_i \cap p M_i = p N_i$ 
(equivalently, $M_i/N_i$ is $p$-torsion free) and the induced maps $\phi_i:N_{i-1}/pN_{i-1}\to N_i/pN_i$
are isomorphisms for all $i\ge1$.  
Define submodules
\[
U_i=\{\omega\in M_i| \phi_{i+s}\circ\phi_{i+s-1}\circ\cdots\circ\phi_{i+1}(\omega)\is0\mod{p^sM_{i+s}}
\ \mbox{for all $s\ge1$}\} \subset M_i. 
\]

Then, for all $i$,
\begin{enumerate}
\item[(i)] $M_i=N_i+U_i$. 
\item[(ii)] $\phi_i(U_{i-1})\subset p U_i$.
\item[(iii)] $\phi_i(M_{i-1})\subset N_i+p U_i$.
\item[(iv)] $N_i\cap U_i=\{0\}$.
\end{enumerate}
\end{proposition}
\begin{proof} The reader may wish to prove this lemma as an exercise or read the proof of~\cite[Proposition 4.5]{DCI}.
\end{proof}

\begin{proof}[Proof of Theorem~\ref{DCI-main}] We apply Proposition~\ref{key-contraction-lemma} to $M_i=\widehat\Omega_{f^{\sigma^i}}(\mu)$ and 
$\phi_i=\cartier$
for all $i\ge0$. For $N_i$ we take the $\Omega^{(1)}_{f^{\sigma^i}}(\mu)$. 
The property $N_i \cap p M_i = p N_i$ clearly holds. Proposition~\ref{Cartier-mod-p} states
that $\phi_i(M_i) \subset N_i + p M_i$ and the matrix of $\phi_i: N_{i-1}/pN_{i-1} 
\to N_i / p N_i$ is given by $HW^{\sigma^{i-1}}(\mu) \mod p$. Its determinant is then congruent to $\det(HW)^{p^i} \mod p$, which is invertible under the assumption in Theorem~\ref{DCI-main}. So the assumptions of Proposition~\ref{key-contraction-lemma} are satisfied.

By Lemma~\ref{Katz-lemma}(ii) we find that $U_0 = \widehat\Omega_f(\mu) \cap d\Omega_{\rm formal}$. Then application of parts (i) and (iv) of Proposition~\ref{key-contraction-lemma} shows that 
\[
\hat\Omega_f(\mu)=\Omega_f^{(1)}(\mu) \oplus U_0
\]
as $R$-modules. Part~(iii) shows that $\cartier(\widehat\Omega_f) \subseteq \Omega_{f^\sigma}^{(1)}(\mu)  +p \, d \Omega_{\rm formal}$ as claimed.
\end{proof}

\subsection{Proof of Gauss congruences}\label{sec:gauss-cong-proof} We are now ready to prove Theorem~\ref{gauss-thm}. Let us recall its formulation. We assume that $f(\v x) \in \Z[x_1^{\pm 1},\ldots,x_n^{\pm 1}]$ has a Newton polytope $\Delta$ whose only integral points are vertices. Then for any Laurent polynomial $h(\v x)$ with coefficients in $\Z$ and $\supp(h) \subset \Delta$ the coefficients of any formal expansion 
\[
\frac{h(\v x)}{f(\v x)} = \sum_{\v v \in \Z^n} c_\v v \v x^{\v v}
\]
satisfy congruences $c_{\v v} \is c_{\v v/p} \mod {p^{\ord_p(\v v)}}$ for every prime $p$ which doesn't divide any of the coefficients of $f(\v x)$.  

The desired congruence is equivalent to the fact that $\omega=h(\v x)/f(\v x)$ satisfies
\[
\cartier \omega - \omega \, \in p \, d\Omega_{\rm formal},
\]
where $d \Omega_{\rm formal}$  is the module of formal derivatives introduces in \S\ref{sec:formally-exact} . It clearly suffices to prove this claim for $\omega=\v x^{\v u}/f(\v x)$ where $\v u$ is a vertex of $\Delta$. Consider
\[
\mu = \Delta \setminus \text{union of all faces which do not contain }\v u.
\] 
This set is \emph{open} in the topology of \S\ref{sec:diff-forms} and its only integral point is $\v u$, that is we have $\mu_\Z = \{\v u\}$. The $1 \times 1$ Hasse--Witt matrix for this $\mu$ is
\[
HW(\mu) = \text{ coefficient of } \v x^{(p-1)\v u} \text{ in } f(\v x)^{p-1} = f_{\v u}^{p-1} \is 1 \mod p.
\]
Here $f_{\v u} \in \Z$ is the coefficients at $\v x^\v u$ in $f(\v x)$ and the congruence holds because $p \nmid f_{\v u}$ by our assumption on $p$. Hence the Hasse--Witt condition holds for $\mu$ over $R=\Z_p$ and by Theorem~\ref{DCI-main} applied with the trivial Frobenius lift $\sigma=id$ there exists a unique $\lambda \in \Z_p$ such that 
\be{gauss-cong-proof-line}
\cartier \omega = \lambda \,  \omega\; \mod {p \, d\Omega_{\rm formal}}.
\ee
By Theorem~\ref{cartier-props} the Cartier operation, and hence this $p$-adic number $\lambda$, is independent of the choice of vertex of $\Delta$ at which the formal expansion is done. Hence we can use expansion at $\v u$ to determine $\lambda$. For this expansion we have $c_{\v 0}=f_{\v u}^{-1} \ne 0$, and therefore comparing the constant terms on the two sides of~\eqref{gauss-cong-proof-line} we find that $\lambda = 1$ as desired.

\section{Periods}\label{sec:periods}

In this lecture we will focus on computation of the Cartier action modulo formal derivatives which were introduced in \S\ref{sec:formally-exact}. The setup is like in \S\ref{sec:diff-forms}-\ref{sec:formally-exact}. We work with a Laurent polynomial $f(\v x) \in R[x_1^{\pm 1},\ldots,x_n^{\pm 1}]$ with coefficients in a characteristic~0 ring $R$ such that $\cap_{s \ge 1} p^s R = \{0\}$. We also assume that $R$ is $p$-adically complete. 

The Newton polytope of $f(\v x)$ is denoted by $\Delta \subset \R^n$. Let $\mu \subseteq \Delta$ be an open set in the sense of \S\ref{sec:diff-forms}. Recall that, if the Hasse--Witt condition holds for $\mu$, Theorem~\ref{DCI-main} implies existence of a unique matrix $\Lambda(\mu) = (\lambda_{\v u, \v v})_{\v u, \v v \in \mu_\Z}$ with entries in $R$ such that
\be{lambda-matrix-def}
\cartier\left(\frac{\v x^{\v u}}{f(\v x)}\right) = \sum_{\v v \in \mu_\Z} \lambda_{\v u, \v v} \, \frac{\v x^{\v v}}{f^\sigma(\v x)} \quad \mod {p \, d \Omega_{\rm formal}}
\ee
for every $\v u \in \mu_\Z$. We are going to exploit divisibility properties of coefficients of elements of $d \Omega_{\rm formal}$ to compute the Cartier entries $\lambda_{\v u, \v v} \in R$. We first explain a straightforward approach to this question, which is a version of Nick Katz's \emph{Internal reconstruction of unit-root F-crystals via expansion coefficients} in~\cite{Katz}. Later we will move to more sophisticated methods using period maps. 

\subsection{$p$-adic interpolation of Cartier matrices via expansion coefficients}\label{sec:cartier-interpolation-on-unit-root}
In Section~\ref{sec:formal-exp} we explained the procedure of formal expansion of rational functions with respect to a vertex $\v b$ of $\Delta$. We have expansions 
\[
\omega = \sum_{\v u \in \Z^n} c_\v u(\omega) \v x^\v u, \qquad \omega \in \Omega_f.
\]

\begin{proposition}\label{Katz-method} For any vector $\v w \in \Z^n$ one can consider the sequence of vectors $A_s \in R^{\#\mu_\Z}$, $s \ge 1$ with entries 
\[
(A_s)_{\v u \in \mu_\Z} = c_{p^s \v w} \left(\frac{\v x^{\v u}}{f(\v x)}\right).
\]
Then 
\[
A_s \is \Lambda(\mu) \, \sigma(A_{s-1}) \quad \mod {p^s}.
\] 
\end{proposition}
\begin{proof}
Let us take the expansion coefficient at $p^{s-1} \v w$ on both sides of~\eqref{lambda-matrix-def}. By Lemma~\ref{Katz-lemma} this coefficient is divisible by $p^{s-1}$ on $d \Omega_{\rm formal}$, and therefore we obtain
\[
c_{p^{s}\v w}\left(\frac{\v x^{\v u}}{f(\v x)}\right)  \is \sum_{\v v \in \mu_\Z} \lambda_{\v u, \v v} \, c_{p^{s-1} \v w}\left(\frac{\v x^{\v v}}{f^\sigma(\v x)}\right)  \; \mod {p^{s}}.
\]
This congruence is precisely our claim. 
\end{proof}
  
Usually one can determine $\Lambda(\mu)$ from congruences as in Proposition~\ref{Katz-method} taken for several exponent vectors $\v w$. 

\subsection{Period maps}

\begin{definition} A \emph{period map} with values in an $R$-module $S$ is a homomorphism of $R$-modules $\sP:\Omega_f \to S$ which vanishes on $d \Omega_f$.
Values of a period map are called~\emph{periods}.
\end{definition}

The classical notion of periods for algebraic varieties was introduced by Alexander Grothendieck. Since elements of $d\Omega_f$ correspond to exact differential $n$-forms on $\T^n \setminus X_f$, there are classical period maps of integration along topological n-cycles $Y \subset (\T^n \setminus X_f)(\C)$:
\[
\sP_{\emph{Y}} : \omega \to \int_{\emph{Y}} \omega. 
\]  
Here we identify elements of $\Omega_f$ with differential forms on the complement $\T^n \setminus X_f$ as was explained in \S\ref{sec:diff-forms}, that is $\frac{h(\v x)}{f(\v x)^m} \mapsto \frac{h(\v x)}{f(\v x)^m}\frac{x_1}{x_1}\ldots\frac{x_n}{x_n}$. Elements of $d \Omega_f$ correspond to exact forms, and therefore $\sP_Y(d\Omega_f)=0$. In general such maps take values in $\C$, or in analytic functions in case of families of hypersurfaces, and we do not expect them to behave well with respect to the Cartier operation. Our goal will be to construct Cartier-invariant period maps. We start with an example.

\begin{example}\label{period0example} Take $f(\v x)=1-t g(\v x)$ with $g(\v x) \in \Z[x_1^{\pm 1}, \ldots, x_n^{\pm 1}]$. Assume that $\v 0 \in \Delta$ (not necessarily a vertex or an internal point). We can work with any ring $R$ such that $\Z[t] \subset R \subset \Z_p\lb t \rb$. Consider the formal expansion of rational functions at $\v 0$:
\[
\frac{h(\v x)}{f(\v x)^m} = h(\v x) \sum_{s \ge 0} \binom{s+m-1}{m-1} t^s g(\v x)^s = \sum_{\v u \in \Z^n} c_{\v u}(t) \v x^\v u.
\]
This expansion is convergent $t$-adically, and the coefficients $c_\v u(t)$ belong to the ring of formal power series $S = \Z_p\lb t \rb$. Then
\[
\sP_{\v 0} : \omega \mapsto c_{\v 0}(\omega)
\]
is a period map with values in $\Z_p\lb t \rb$ which
\begin{itemize}
\item vanishes on $d\Omega_{\rm formal}$,
\item is $\cartier$-invariant.
\end{itemize}
The first property is clear because the constant terms of logarithmic derivatives $x_i\frac{\partial }{\partial x_i} \nu$ are zero. The second property follows from the fact that the Cartier operation acts on formal expansions at $\v 0$ is the same way as on expansions at vertices of the Newton polytope. That is, we have $\cartier: \sum_\v u c_\v u(t) \v x^\v u \mapsto \sum_\v u c_{p \v u}(t) \v x^\v u$. The reader who is not convinced may find a discussion of this later point in~\cite[\S 2]{DCII}. 
\end{example} 

We note that the period map in the above example is of the kind of earlier mentioned integration maps. This is integration along the $n$-cycle $Y= S^1 \times \cdots \times S^1$, if we divide the result by $(2 \pi i)^n$, or the residue map at $\v 0$. One can write this fact as $\sP_\v 0 = \frac1{(2 \pi i)^n} \sP_{S^1 \times \cdots \times S^1}$.

Now we can revisit \S\ref{sec:dwork-cong}. We prove the following fact.

\begin{theorem}\label{frob-unit-root-p-adic-analytic-thm}  Take $f(\v x)=1-t g(\v x)$ with $g(\v x) \in \Z[x_1^{\pm 1}, \ldots, x_n^{\pm 1}]$. Assume that the Newton polytope $\Delta$ has only one interior integral point. Assume further that this interior point is the origin, that is $\Delta^\circ \cap  \Z^n = \{\v 0 \}$. For $i \ge 0$ we denote by $c_i \in \Z$ the constant term of $g(\v x)^i$. Consider the generating series 
\[
\gamma(t)=\sum_{i=0}^\infty c_i t^i = \frac1{(2 \pi i)^n}\oint \ldots \oint \frac1{1-t g(\v x)}\frac{dx_1}{x_1}\ldots \frac{dx_n}{x_n}
\] 
and denote its truncations by $\gamma_{m}(t)=\sum_{i=0}^{m-1}c_i t^i$. Then for any odd prime number $p$ one has
\[
\frac{\gamma(t)}{\gamma(t^p)} \in \Z[t,\gamma_{p}(t)^{-1}]\,\hat{\;}.
\] 
\end{theorem} 
\begin{proof} Consider $\mu=\Delta^\circ$. The respective Hasse-Witt matrix (of rank~1 in this case) is given by $HW(t) = \text{ constant term of } (1 - t g(\v x))^{p-1} = \sum_{i=0}^{p-1} (-1)^i \binom{p-1}{i} c_i t^i$. Since $HW(t) \is \gamma_{p}(t) \mod p$, it follows that $HW(t)$ is invertible in the ring $R=\Z[t,\gamma_{p}(t)^{-1}]\,\hat{\;}$. Let $\sigma: R \to R$ be the Frobenius lift $\sigma(r(t))=r(t^p)$. By Theorem~\ref{DCI-main} there exists $\lambda \in R$ such that 
\be{lambda-in-dwork-cong}
\cartier \frac1{f(\v x)} = \lambda  \frac1{f^\sigma(\v x)} \mod { p\, d \Omega_{\rm formal}}. 
\ee
We apply to this congruence the period map $\sP_{\v 0}$ defined in Example~\ref{period0example}. Note that $\sP_\v 0(1/f(\v x)) = \gamma(t)$. Using the two properties of the map $\sP_{\v 0}$ we get
\[
\gamma(t) = \lambda \, \gamma(t^p),
\]
from which it follows that $\lambda = \gamma(t)/\gamma(t^p)$. Since $\lambda \in R$, this proves our claim.
\end{proof}

Our next goal will be to prove an explicit $p$-adic approximation to $\gamma(t)/\gamma(t^p)$ by ratios of truncations $\gamma_{p^s}(t)/\gamma_{p^{s-1}}(t^p)$. For that we should evidence that truncations are periods. In fact, they will be periods modulo $p^s$ which we shall now introduce. 

\subsection{Periods modulo $p^s$}\label{sec:periods-mod-p-s}

\begin{definition} Let $S$ be an $R$-module such that $\cap_s p^s S=\{0\}$. An $R$-linear map $\rho: \Omega_f \to S$ such that $\rho(d \Omega_f) \subset p^s S$ is called a period map modulo $p^s$. 
\end{definition}

\begin{example}\label{exp-coeffs-as-periods} Let $\omega \mapsto \sum_{\v u \in \Z^n} c_{\v u}(\omega) \v x^\v u$ be either one of the procedures of formal expansion introduced earlier. It can be either an expansion at a vertex $\v b \in \Delta$ as in \S\ref{sec:formal-exp} or an expansion at $\v 0 \in \Delta$ as in Example~\ref{period0example}. We take $S=R$ and $S=\Z_p\lb t \rb$ in these cases respectively. Then for any exponent vector $\v w \in \Z^n$ the map \[
\omega \mapsto c_{p^s \v w}(\omega)
\]
is a period map modulo $p^s$ with the following properties:
\begin{itemize}
\item $c_{p^s \v w}\left(d\Omega_{\rm formal}\right) \subset p^s S$,
\item $c_{p^s \v w} = c_{p^{s-1} \v w} \circ \cartier$.
\end{itemize}
\end{example}

\begin{example}\label{generalized-exp-coeffs} Let $\omega \mapsto \sum_{\v u \in \Z^n} c_{\v u}(\omega) \v x^\v u$ be either one of the procedures of formal expansion introduced earlier with coefficients in $S=R$ or $S=\Z_p\lb t \rb \supset R$.  Let $\ell(\v x) \in S[x_1^{\pm 1},\ldots,x_n^{\pm 1}]$ be a Laurent polynomial. Then
\[
\rho_{s,\ell} : \omega \mapsto c_{\v 0}\left( \ell(\v x)^{p^s} \omega\right) 
\]
is a period map modulo $p^s$ with the following properties:
\begin{itemize}
\item $\rho_{s,\ell}\left(d\Omega_{\rm formal}\right) \subset p^s S$,
\item $\rho_{s,\ell} \is \rho_{s-1,\ell^\sigma} \circ \cartier \; \mod {p^s}$.
\end{itemize}
The reason that the first property holds is that 
\[
\ell(\v x)^{p^s} x_i\frac{\partial }{\partial x_i} \nu \is  x_i\frac{\partial }{\partial x_i} \left( \ell(\v x)^{p^s} \nu\right) \mod {p^s}.
\]
We than take $c_{\v 0}$ and obtain that $\rho_{s,\ell}(x_i\frac{\partial }{\partial x_i} \nu) \in p^s S$.
Proof of the second property is left an an exercise for the reader.
\end{example}

The fact that expansion coefficients are period maps (Example~\ref{exp-coeffs-as-periods}) was actually used in the proof of Proposition~\ref{Katz-method}. Example~\ref{generalized-exp-coeffs} provides a generalisation, which we shall now use to prove Dwork's congruences announced in \S\ref{sec:dwork-cong}.

\begin{exercise}\label{Dwork-cong-exercise} Prove Theorem~\ref{dwork-congs-them}. For that, check that $\gamma_{p^s}(t)=(\rho_{s,1}-\rho_{s,t g(\v x)})(1/f(\v x))$ and apply these period maps to the identity~\eqref{lambda-in-dwork-cong}.
\end{exercise}  

Now we are in a position to prove the following generalisation of Theorem~\ref{HHW-main}. 

\begin{theorem}\label{DCI-Dwork-cong-thm} Let $\mu \subseteq \Delta$ be an \emph{open} set. We consider the sequence of square matrices 
\[
\beta_m(\mu)_{\v u, \v v \in \mu_\Z} = \text{ coefficient of } \v x^{m \v v - \v u} \text { in } f(\v x)^{m-1}.
\]
Suppose that the Hasse--Witt condition holds for $\mu)$, that is matrix $\beta_p(\mu) = HW(\mu)$ is invertible, and let $\Lambda(\mu)$ be the Cartier matrix~\eqref{lambda-matrix-def} on the quotient of $\Omega_f(\mu)$ by the submodule of formal derivatives. Then for every $s \ge 1$ we have
\[
\beta_{m p^s}(\mu) \is \Lambda(\mu) \sigma(\beta_{m p^{s-1}}(\mu)) \;\mod {p^s}. 
\]
In particular, when $m=1$,
\[
\Lambda(\mu) \is \beta_{p^s}(\mu) \circ \sigma(\beta_{p^{s-1}}(\mu))^{-1} \;\mod {p^s}.
\]
\end{theorem}
\begin{proof} Denote $\omega_\v u = \v x^\v u / f(\v x)$ for $\v u \in \mu_\Z$. Matrix $\Lambda(\mu) = (\lambda_{\v u, \v v})$  satisfies
\[
\cartier(\omega_\v u) \is \sum_{\v v \in \mu_\Z} \lambda_{\v u, \v v} \omega_\v v \; \mod {p d\Omega_{\rm formal}}.
\]
For $\v w \in \mu_\Z$ let $\ell(\v x)= f(\v x)^{m}/\v x^{m\v w}$.  We apply to the above congruence the period map $\rho_{s-1, \ell^\sigma}$ defined in Example~\ref{generalized-exp-coeffs}. Using the two properties of this period map we get
\[
\rho_{s, \ell}(\omega_\v u) \is \sum_{\v v \in \mu_\Z} \lambda_{\v u, \v v} \, \rho_{s-1, \ell^\sigma}(\omega^\sigma_{\v v}) \; \mod {p^s}.
\]  
It remains to notice that 
\[
\rho_{s, \ell}(\omega_\v u) = c_{\v 0} \left( f(\v x)^{p^s m-1} \v x^{\v u - p^s m\v w} \right) = \beta_{p^s m}(\mu)_{\v u, \v w}.
\]
\end{proof} 

\section{Beyond the unit root part}

So far we developed a method which gives those Frobenius roots which are $p$-adic units. In the last lecture we will introduce \emph{higher Hasse-Witt conditions} and explain a far going generalization of the previous  results. These methods will allow to construct Cartier matrices on ($p$-adic completions of) the whole de Rham cohomology modules. We will also discuss applications and related phenomena of \emph{supercongruences}. 

The setup is like in \S\ref{sec:diff-forms}-\ref{sec:cartier}. We work with a Laurent polynomial $f(\v x) = \sum_{\v u} f_{\v u}\v x^\v u \in R[x_1^{\pm 1},\ldots,x_n^{\pm 1}]$ with coefficients in a characteristic~0 ring $R$ such that $\cap_{s \ge 1} p^s R = \{0\}$. We also assume that $R$ is $p$-adically complete and equipped with a Frobenius lift $\sigma: R \to R$. 

\subsection{Higher formal derivatives}

In \S\ref{sec:formal-exp} and \S\ref{sec:formally-exact} we defined an embedding of $\Omega_f$ into the space of formal power series 
\[
\Omega_{\rm formal} = \{ \sum_{\v u \in C(\Delta - \v b)\cap \Z^n} a_\v u \v x^\v u \;|\; a_\v u \in R \}
\]
by the procedure of formal expansion of rational functions at a vertex $\v b \in \Delta$. It is assumed that $f_{\v b} \in R^\times$.

\begin{definition} The submodule of $k$-th formal derivatives\[
d^k\Omega_{\rm formal} \subset \Omega_{\rm formal}
\]
is the $R$-module generated by elements of the form $\theta_{i_1}\cdots\theta_{i_k}\nu$, where $\theta_i=x_i\frac{\partial}{\partial x_i}$, $\nu\in\Omega_{\rm formal}$ and $1 \le i_j \le n$ for each index $j=1,\ldots,k$.  
\end{definition}

We have the following generalisation of Lemma~\ref{Katz-lemma}:

\begin{lemma}\label{Katz-lemma-k} A series $\nu=\sum_{\v u} a_\v u \v x^\v u \in \Omega_{\rm formal}$ is a $k$th formal derivative if and only if one of the following equivalent conditions holds
\begin{itemize}
\item[(i)] $a_\v u \in g.c.d.(u_1,\ldots,u_n)^k R$ for each $\v u$,
\item[(ii)] $\cartier^s(\nu) \in p^{ks} \Omega_{\rm formal}$ for each $s \ge 1$.
\end{itemize}
\end{lemma}
\begin{proof}The proof is left as an exercise.\end{proof}

\subsection{Higher Hasse--Witt conditions}

Let $\mu \subset \Delta$ be an \emph{open} set in the finite topology introduced in \S\ref{sec:diff-forms}. There is a natural filtration on $\Omega_f(\mu )$ by the order of pole along the zero locus of $f(\v x)$: 
\[
\Omega_f^{(k)}(\mu) = {\rm Span}_R \left( (k-1)!\frac{\v x^{\v u}}{f(\v x)^k} \right)_{\v u\in (k \mu)_\Z}, \quad k \ge 1.
\]
On de Rham cohomology this filtration corresponds to the Hodge filtration, see Remark~\ref{geometric-remark}.
The main result of~\cite{DCIII} describes a splitting filtration in arithmetic terms: 

\begin{theorem}\label{DCIII-main} Let $R$ be $p$-adically complete, $\mu \subseteq \Delta$ be an \emph{open} set and $1 \le k < p$. If the \emph{$k$th Hasse--Witt condition} holds for $\mu$, then there is a direct sum decomposition of $R$-modules
\[
\hat\Omega_f(\mu) = \Omega_f^{(k)}(\mu) \oplus \sF_k, 
\]
where 
\[
\sF_k = \hat\Omega_f(\mu) \cap d^k\Omega_{formal} = \{ \omega \in \hat\Omega_f(\mu) : \cartier^s(\omega) \in p^{ks} \hat\Omega_{f^{\sigma^s}}(\mu) \quad \forall s \ge 1\}
\] is the submodule of $k$th formal derivatives. Secondly, one has
\[
\cartier(\Omega_f(\mu)) \subset \Omega^{(k)}_{f^\sigma}(\mu) + p^k \sF_k^\sigma.
\]  
\end{theorem}

\bigskip

This theorem will be proved in the next section. Let us recall formula~\eqref{cartier-formula} which describes explicitly the Cartier action $\cartier: \hat\Omega_f \to \hat\Omega_{f^\sigma}$:
\be{cartier-formula-again}
\cartier\left((m-1)! \frac{\v x^\v u }{f(\v x)^m}\right)=\sum_{r\ge0}\frac{p^r}{r!}
\frac{(m-1)!}{(\lceil m/p\rceil-1)!}
(\lceil m/p\rceil+r-1)!\frac{Q_r(\v x)}{f^\sigma(\v x)^{r+\lceil m/p\rceil}},
\ee
where the $Q_r(\v x)=\cartier(G(\v x)^r\v x^{\v u}f(\v x)^{p\lceil m/p\rceil-m})$ are
Laurent polynomials in $x_1,\ldots,x_n$ with support in $(\lceil m/p\rceil+r)\Delta$ and coefficients in $R$.  
Polynomial $G(\v x)$ was defined as $(f^\sigma(\v x^p)-f(\v x)^p)/p$, it is supported in $p \Delta$ and has coefficients in $R$. We observe that formula~\eqref{cartier-formula-again} implies that 
\[
\cartier(\Omega_f(\mu)) \subset \Omega_{f^\sigma}^{(k)}(\mu) + p^k \hat\Omega_{f^\sigma}(\mu).
\]
The $k$th Hasse-Witt condition means that the image of the Cartier operator modulo $p^k$ has maximal possible rank. We shall now explain how to verify this condition in practice.
 
\begin{definition}\label{higherHasseWitt} Assume that $1 \le k < p$. Denote
\[
F^{(k)}(\v x) = f(\v x)^{p-k}\sum_{r=0}^{k-1}\left(f^\sigma(\v x^p)-f(\v x)^p\right)^r
f^{\sigma}(\v x^p)^{k-1-r}.
\]
The $k$-th Hasse-Witt matrix $HW^{(k)}$ is the matrix indexed by the set $(k \Delta)_\Z = (k \Delta) \cap \Z^n$ with entries given by 
\[
HW^{(k)}_{\v u,\v v}=\mbox{coefficient of $\v x^{p\v v-\v u}$ in } F^{(k)}(\v x).
\]
For an \emph{open} set $\mu \subseteq \Delta$ we denote by $HW^{(k)}(\mu)$ the submatrix indexed by $(k \mu)_\Z$.  
\end{definition}

\begin{exercise}
Check that 
\[HW^{(k)}_{\v u,\v v}\is\text{coefficient of $\v x^{p\v v-\v u}$ in }
\frac{f^\sigma(\v x^p)^k}{f(\v x)^k}\mod{p^k},
\]
where on the right one has to consider the Laurent series expansion in order to determine the coefficient.
\end{exercise}

\begin{lemma}\label{cartier-vs-HW-lemma} Let $\mu \subset \Delta$ be an \emph{open} set. \begin{itemize}
\item[(i)] For $\v u \in (k\mu)_\Z$ one has \[
\cartier\left(\frac{\v x^\v u}{f(\v x)^k}\right) \is \sum_{\v v \in (k \mu)_\Z} HW^{(k)}_{\v u,\v v} \frac{\v x^\v v}{f^\sigma(\v x)^k} \quad \mod {p^k \, \hat \Omega_{f^\sigma}(\mu)} 
\]

\item[(ii)] Let us denote $m_{\ell}=\# (\ell \mu)_\Z$ for $\ell \ge 1$ and
\[
L(k,\mu) = \sum_{\ell = 1}^k (\ell-1) (m_{\ell}-m_{\ell-1}).
\] 
Then  $\det \left(HW^{(k)}(\mu)\right) \in p^{L(k,\mu)} R$.
\end{itemize}
\end{lemma}
\begin{proof} (i) follows from formula~\eqref{cartier-formula-again}. We leave the check as an exercise.
For (ii) we choose a basis $\omega_1,\ldots,\omega_{m_k}$ in the free $R$-module $\Omega^{(k)}_f(\mu)$ so that $\omega_1,\ldots,\omega_{m_\ell}$ is a basis in $\Omega^{(\ell)}_f(\mu)$ for every $\ell \le k$. This is called an \emph{extended basis}. Under our standard assumption that there is a vertex $\v b$ in $\Delta$ such that the coefficient of $f$ at $\v x^\v b$ is in $R^\times$, one may check that  
\[
\frac{\v x^{\v u}}{f(\v x)^\ell}, \v u \in (\ell \mu)_\Z \setminus ((\ell-1) \mu)_\Z, \quad 1 \le \ell \le k
\]   
is an extended basis in $\Omega_f^{(k)}(\mu)$.

We write the action of the Cartier operator on $\Omega_f^{(k)}(\mu)$ modulo $p^k$ with respect to an extended basis. That is, we choose some $C_{ij} \in R$ such that $\cartier(\omega_i) = \sum_j C_{ij} \omega_j^\sigma \mod {p^k}$.  Formula~\eqref{cartier-formula-again} shows that $\cartier$ maps $\Omega_f^{(k)}(\mu)$ to $\sum_{l=1}^{k}p^{l-1}\Omega_{f^\sigma}^{(l)}(\mu)$, and hence $C_{ij}$ is divisible by $p^{l-1}$ when 
$m_{l-1}<j\le m_\ell$. We conclude that $\det(C) \in p^{L(k,\mu)}R$. 

Let $\tilde\omega_{\v u} = \v x^{\v u}/f(\v x)^k$, $\v u \in (k \mu)_{\Z}$ be the monomial basis in $\Omega_f^{(k)}(\mu)$. From part~(i) we know that $HW^{(k)}(\mu)$ describes the action of the Cartier operator modulo $p^k$ in this monomial basis. Let $A$ be the transition matrix, that is $\tilde\omega_{\v u} = \sum_i A_{\v u,i} \omega_i$. Then 
\[
C \is A^{-1} \cdot HW^{(k)} \cdot A^\sigma \qquad \mod {p^k}.
\] 
It follows that $\det(A)^{-1} \det(HW^{(k)}(\mu)) \det(A)^\sigma \in p^{L(k,\mu)}R$. Since $R$ is $p$-adically complete, invertibility $\det(A) \in R^\times$ implies that $\det(A)^\sigma \in R^\times$. It follows that 
\[
\det(HW^{(k)}(\mu)) \in p^{L(k,\mu)}R.
\]
\end{proof}

\begin{definition} We say that the $k$th Hasse-Witt condition holds for an \emph{open} subset $\mu \subseteq \Delta$ when
\be{cartier-image-mod-p-k} 
\det(HW^{(\ell)}(\mu)) \in p^{L(\ell, \mu)}R^{\times}, \quad 1 \le \ell \le k.
\ee
\end{definition}

In~\cite[\S 5]{DCIII} it is explained that this condition means maximality of the Cartier image modulo $p^k$.

\subsection{Proof of the main theorem} In this section we will prove Theorem~\ref{DCIII-main}. The proof was given in~\cite[\S 4-5]{DCIII} in a more general situation of \emph{Dwork crystals}. Here we will give a somewhat simplified version, which works for modules $\Omega_f(\mu)$ and exploits our key contraction principle stated in Proposition~\ref{key-contraction-lemma}.   

To shorten the notation we shall drop $\mu$ from $\Omega_f(\mu)$ and $L(\ell,\mu)$ for the duration of this proof. Similarly, $\Omega_f^{(k)}$ will stand for $\Omega_f^{(k)}(\mu)$. We thus need to show that the $k$th Hasse-Witt condition implies that $\hat\Omega_f = \Omega_f^{(k)} \oplus \fil_k$ and $\cartier(\hat\Omega_f) \subset \Omega_{f^\sigma}^{(k)} + p^k \fil_k^\sigma$, where $\fil_k^\sigma$ is the submodule of $k$th formal derivatives in $\hat\Omega_{f^\sigma}$. 
 
We will use induction on $k$. The case $k=1$ is already handled in Theorem~\ref{DCI-main}. Suppose now that $k>1$ and 
\be{induction-assum}
\hat\Omega_f\cong\Omega_f^{(l)}\oplus \fil_{l}\quad\mbox{for all $l<k$}. 
\ee
We will show that 
\[
\fil_{k-1}\cong(\Omega_f^{(k)}\cap\fil_{k-1})\oplus \fil_k,
\]
and hence
\[
\hat\Omega_f\cong\Omega_f^{(k-1)}\oplus(\Omega_f^{(k)}\cap\fil_{k-1})\oplus \fil_k
\cong\Omega_f^{(k)}\oplus\fil_k.
\]
The last equality is a consequence of \eqref{induction-assum} with $l=k-1$
restricted to $\Omega_f^{(k)}$.  Our convention is that $\fil_0=\hat\Omega_f$.

We would like to apply Proposition~\ref{key-contraction-lemma} with 
\[
M_i = \fil_{k-1}^{\sigma^i}, \; \phi_i = p^{1-k}\cartier \text{ and } 
N_i = M_i\cap\Omega_{f^{\sigma^i}}^{(k)}.
\]
We shall check that the assumptions of Proposition~\ref{key-contraction-lemma} are satisfied. It is clear from the definition of modules $\fil_{\ell}$ that $\cartier(\fil_{\ell}) \subset p^{\ell} \fil_{\ell}^\sigma$ for any $\ell$. Therefore $\phi_i$ maps $M_{i-1}$ to $M_i$. 
Let $\omega\in\fil_{k-1}$. We know that $\cartier(\omega)\in p^{k-1}\fil^\sigma_{k-1}$.
Since $\cartier$ maps $\hat\Omega_f$ modulo $p^k$ to $\Omega_{f^\sigma}^{(k)}$ we can conclude
that $\cartier(\omega) = p^{k-1} \omega_1 + p^k \omega_2$ with 
$\omega_1 \in \fil^\sigma_{k-1}\cap\Omega_{f^\sigma}^{(k)}$, 
$\omega_2 \in \hat\Omega_{f^\sigma}$.
Using~\eqref{induction-assum} for $f^\sigma$ we write $\omega_2 = \omega_2' + \nu_2$
with $\nu_2 \in \fil^{\sigma}_{k-1}$ and $\omega_2' \in \Omega_{f^\sigma}^{(k-1)}\subset
\Omega_{f^\sigma}^{(k)}$.
We get $\cartier(\omega) = p^{k-1}\nu_1 + p^k \nu_2$ with $\nu_1 = \omega_1+p \omega_2'
\in \Omega_{f^\sigma}^{(k)}$. Note that $\nu_1 = p^{1-k}\cartier(\omega)-p\nu_2 
\in \fil_{k-1}^\sigma$.
We thus proved that
\[
\cartier(\fil_{k-1}) \subset p^{k-1}(\fil_{k-1}^\sigma\cap\Omega_{f^\sigma}^{(k)}) + 
p^k \fil_{k-1}^\sigma.
\]
Replacing $f$ with $f^{\sigma^i}$ and dividing by $p^{k-1}$, we then obtain 
\be{assum-1}
\phi_i (M_{i-1}) \subset N_i + p M_i.
\ee

Next, we need to check that 
\be{assum-2} 
N_i \cap pM_i = pN_i.
\ee 
We will give the argument in the case $i=0$; other cases will follow by replacing 
$f$ with $f^{\sigma^i}$. Recall that $N_0\cap pM_0=\fil_{k-1}\cap\Omega_f^{(k)}\cap p\fil_{k-1}
=\Omega_f^{(k)}\cap p\fil_{k-1}$. One easily checks that this equals 
$p(\Omega_f^{(k)}\cap\fil_{k-1})=pN_0$.

Note that by~\eqref{assum-1} and~\eqref{assum-2} we have induced maps
\be{assum-3-maps}
\overline \phi_i:N_{i-1}/pN_{i-1} \to N_i/p N_{i}.
\ee
To apply Proposition~\ref{key-contraction-lemma}, it remains to check that~\eqref{assum-3-maps} are isomorphisms. We shall restrict to the case $i=1$, the other cases being similar.

Let us denote $m_{\ell}=\# (\ell \mu)_\Z$. Let $\omega_1,\ldots,\omega_{m_k}$ be a basis of $\Omega_f^{(k)}$ such that $\omega_1,\ldots,\omega_{m_\ell}$ is a basis in $\Omega^{(\ell)}_f(\mu)$ for every $\ell \le k$. This is called an \emph{extended basis}. Its existence is a mild assumption (see~\cite[\S 5]{DCIII}). For example, if there is a vertex $\v b$ in $\Delta$ such that the coefficient of $f$ at $\v x^\v b$ is in $R^\times$, then an extended basis exists (the reader can find in the proof of Lemma~\ref{cartier-vs-HW-lemma}). 

For $l=k-1$ down to $l=1$ and $r=m_\ell+1,\ldots,m_{\ell+1}$ we choose
$\eta_r\in\Omega_f^{(l)}$ such that $\omega_r-\eta_r\in\fil_l$. This is possible because
of~\eqref{induction-assum}. Then redefine $\omega_r:=\omega_r-\eta_r$. We now have a new extended basis of $\Omega_f^{(k)}$ such that 
\[
\omega_r\in\fil_l\cap\Omega_f^{(l+1)} \text{ whenever } m_\ell < r \le m_{\ell+1}.
\]
In $\Omega_{f^\sigma}^{(k)}$ we can choose
a similar basis $\omega_i^\sigma,i=1,\ldots,m_k$. 

Let $C=(c_{i,j})$ be a matrix of $\cartier$ modulo $p^k$ in these respective bases. That is, we choose some $\lambda_{ij} \in R$ such that
\[
\cartier (\omega_i) = \sum_{j=1}^{m_k} c_{ij} \, \omega^\sigma_{j}  \quad
\mod {p^k \hat\Omega_{f^\sigma}}
\] 
for every $1 \le i \le m_k$. Since $\cartier$ is divisible by $p^\ell$ on $\fil_\ell$, we have $p^{l}|c_{ij}$ when $m_\ell < i \le m_{\ell+1}$. {\bf Exercise}: check that decompositions~\eqref{induction-assum} imply that $c_{ij} \is 0 \mod {p^k}$ when $j < m_\ell \le i$ for some $\ell$. Thus w.l.o.g. we may assume that $C$ is a block-upper-triangular matrix; we will denote its diagonal blocks by $p^{\ell} C_\ell$, $0 \le \ell \le k-1$:
\[
C = \begin{pmatrix} C_0 & * & * & \ldots \\ 0 & p C_1  & p * &  \ldots \\
&&& \\
0 & 0 & \ldots & p^{k-1} C_{k-1} \end{pmatrix}.
\] 
Note that $C_{k-1} \mod p$ is the matrix of the map
\[
p^{1-k} \cartier: \fil_{k-1}\cap\Omega_f^{(k)}\mod{p}\to 
\fil_{k-1}^\sigma\cap\Omega_{f^\sigma}^{(k)}\mod{p}.
\]  
in the bases $\omega_i,\omega^\sigma_i,m_{k-1}< i \le m_k$. This map is precisely~\eqref{assum-3-maps} for $i=1$. We can now conclude that this map is invertible if and only if 
\[
\det(C_{k-1}) \in R^\times.
\] 

Recall that $HW^{(k)}$ is the matrix of $\cartier \mod {p^k}$ in the monomial basis $\tilde\omega_{\v u}=\v x^{\v u}/f(\v x)^k$, $\v u \in (k \mu)_\Z$ (see Lemma~\ref{cartier-vs-HW-lemma}(i)). Let $A$ be the transition matrix, that is $\tilde\omega_{\v u} = \sum_i A_{\v u,i} \omega_i$. Then 
\be{extended-vs-monomial-HW}
C \is A^{-1} \cdot HW^{(k)} \cdot A^\sigma \qquad \mod {p^k}.
\ee
It follows that $\det(A)^{-1} \det(HW^{(k)}(\mu)) \det(A)^\sigma \in p^{L(k,\sigma)}R$. Since $R$ is $p$-adically complete, invertibility $\det(A) \in R^\times$ implies that $\det(A)^\sigma \in R^\times$. It follows that $\det(HW^{(k)}(\mu)) \in p^{L(k,\sigma)}R$.

Let $\delta$ be the diagonal matrix of size $m_k$ whose $j$-th entry equals
$p^{l-1}$ where $m_{\ell-1}<j\le m_{\ell}$. Since the $j$-th row in~\eqref{extended-vs-monomial-HW} is 
divisible by $p^{l}$, we can multiply~\eqref{extended-vs-monomial-HW} by the matrix $\delta^{-1}$ on the left retaining a congruence mod $p$ between matrices with entries in $R$. Note that $\det \delta = p^{L(k)}$. We conclude that
\[
\prod_{\ell = 0}^{k-1} \det (C_{\ell}) \is p^{-L(k)} \det(A)^{-1} \det(HW^{(k)}) \det(A^{\sigma}) \qquad \mod p. 
\]
The Hasse--Witt condition yields $p^{-L(k)}\det(HW^{(k)})\in R^\times$. Since $R$ is $p$-adically complete, $\det(A) \in R^{\times}$ implies that $\det(A)^\sigma \in R^{\times}$. Hence $\det( C_{k-1}) \in R^\times$,
which concludes the proof of invertibility of~\eqref{assum-3-maps} when $i=1$. 
For general $i$ we use the same argument with $f$ substituted by $f^{\sigma^{i-1}}$. 

We checked that the assumptions of Proposition~\ref{key-contraction-lemma} are satisfied. 
From parts (i) and (iv) of that proposition we conclude that $M_0 = N_0 \oplus U_0$, where

\[\bal
U_0 &= \{ \omega \in M_0 | \phi_{s} \circ \ldots \circ \phi_1 (\omega) \in p^s M_s 
\text{ for all } s \ge 1 \} \\
&= \{ \omega \in \fil_{k-1} | \cartier^s(\omega) \in p^{ks} \fil_{k-1}^{\sigma^s} 
\text{ for all } s \ge 1 \} = \fil_{k}.
\eal\]
We proved that $ \fil_{k-1}\cong (\fil_{k-1}\cap\Omega_f^{(k)}) \oplus \fil_{k}$.
Using the argument from the beginning of the induction step, the first claim of our theorem follows.

Our second claim follows from~\eqref{cartier-image-mod-p-k} 
and the decomposition that we already proved:
\[\bal
\cartier \left( \Omega_f\right) \subset \Omega_{f^\sigma}^{(k)} + p^{k} \hat \Omega_{f^\sigma}
&= \Omega_{f^\sigma}^{(k)} + p^{k} \left( \Omega_{f^\sigma}^{(k)} + \fil_{k}^\sigma\right) \\
&= \Omega_{f^\sigma}^{(k)} + p^k \fil_{k}^\sigma.
\eal\]
This finishes the proof of the theorem.

\subsection{$p$-adic interpolation of higher Cartier matrices}
In \S\ref{sec:cartier-interpolation-on-unit-root} we explained interpolation of Cartier matrices on the unit-root quotients $\Omega_f(\mu)/\fil_1$ via expansion coefficients, the principle which goes back to Nick Katz. This can be done for the quotients by higher derivatives in a similar vein. Let $\omega_1,\ldots,\omega_{\#(k \mu)_\Z}$ be a basis in the free $R$-module $\Omega^{(k)}_f(\mu)$. If the $k$th Hasse--Witt condition holds, then by Theorem~\ref{DCIII-main} there is a unique matrix $\Lambda = (\lambda_{ij})$ with entries in $R$ such that 
\be{cartier-matrix-k}
\cartier(\omega_i) \is \sum_{j=1}^{\# (k \mu)_\Z} \lambda_{ij} \, \omega_j^\sigma \qquad \mod{ p^k \fil_k^\sigma}.
\ee
Consider any procedure of formal expansion $\omega = \sum_{\v u \in \Z^n} c_\v u(\omega) \v x^\v u$ for $\omega \in \Omega_f$. This can be expansion with respect to a vertex $\v b$ of $\Delta$ as in Section~\ref{sec:formal-exp}, or expansion at the origin $\v 0$ as in Example~\ref{period0example}. For any vector $\v u \in \Z^n$, taking expansion coefficients at $p^{s-1} \v u$ in~\eqref{cartier-matrix-k} yields
\[
c_{p^{s}\v u}(\omega_i) \is \sum_{j=1}^{\# (k \mu)_\Z} \lambda_{ij} \, c_{p^{s-1} \v u}(\omega_j^\sigma) \; \mod {p^{s k}}, \quad s \ge 1. 
\]  
In practice one can often \emph{solve systems of such $p$-adic congruences} and obtain explicit expressions for $\Lambda$ in terms of expansion coefficients. In the following sections we will overview several applications in which explicit information about entries of $\Lambda$ appeared to be useful.

\subsection{Supercongruences and excellent Frobenius lifts}\label{sec:exc-lifts}

Let us start with a simple example, in which one can explicitly compute the action of $\cartier$ modulo $\fil_2$ using the method of the previous section. This example will help us to illustrate the phenomenon in the name of this section.   

\begin{example} Let $f(\v x) = (1-x_1)(1-x_2)- t x_1 x_2$ and
\[
R = \text{$p$-adic completion of } \; \Z[t,1/t].
\]
Frobenius lifts $\sigma: R \to R$ are given by $t \mapsto t^\sigma \in t^p + pR$. In \cite[\S 6]{DCIII} we show that for any Frobenius lift one has
\be{simple-example}\bal
\cartier\left(\frac{1}{f(\v x)} \right) &= \frac{1}{f^\sigma(\v x)} \mod {p \fil_1^\sigma}, \\
\cartier\left(\frac{1}{f(\v x)} \right) &= \frac{1}{f^\sigma(\v x)} + \log\left(\frac{t^\sigma}{t^p} \right)\left( \theta \frac{1}{f(\v x)} \right)^\sigma \mod {p^2 \fil_2^\sigma}, \\
\eal\ee
where $\theta= t \frac {d}{dt}$. We see that in the special case $t^\sigma = t^p$ one has 
\be{simple-example-exc-lift}
\cartier\left(\frac{1}{f(\v x)} \right) = \frac{1}{f^\sigma(\v x)} \quad \mod {p \fil_2^\sigma}.
\ee For this special Frobenius lift $1/f(\v x)$ becomes an `eigenvector' of the Cartier operator modulo $\fil_2$, while in general it is only an 'eigenvector' modulo $\fil_1$. Existence of such lifts was observed by Bernard Dwork who called them \emph{excellent Frobenius lifts}.
\end{example}

In the above example, let us expand $1/f$ at the vertex $\v 0 \in \Delta$:
\[\bal
\frac1{f(\v x)} &= \sum_{m \ge 0} \frac{(t x_1 x_2)^m}{(1-x_1)^{m+1}(1-x_2)^{m+1}} = \sum_{m,u_1,u_2 \ge 0} t^m \binom{u_1}{m}\binom{u_2}{m} x_1^{u_1} x_2^{u_2}\\
&=  \sum_{\v u \in \Z_{\ge 0}^2} c_{\v u}(t) \, \v x^\v u, \qquad c_{\v u}(t) = \sum_{m=0}^{\min(u_1,u_2)} \binom{u_1}{m} \binom{u_2}{m} t^m
\eal\]
The first line in~\eqref{simple-example} means that for any Frobenius lift $\sigma: R \to R$
\[
c_{u_1 p^s, u_2 p^s}(t) \is c_{u_1 p^{s-1}, u_2 p^{s-1}}(t^\sigma) \quad \mod {p^s}, \quad s \ge 1.
\]
However for the excellent Frobenius lift $t^\sigma=t^p$ we have~\eqref{simple-example-exc-lift} and the modulus improves:
\[
c_{u_1 p^s, u_2 p^s}(t) \is c_{u_1 p^{s-1}, u_2 p^{s-1}}(t^p) \quad \mod {p^{2s}}, \quad s \ge 1.
\]
This is an example of~\emph{supercongruences}. 

Note that values $t=\pm 1$ are fixed by our excellent Frobenius lift $t \mapsto t^p$. For example, for $t=1$ we have supercongruences for the expansion coefficients of
\[
\frac1{1-x_1-x_2} = \sum_{\v u \in \Z_{\ge 0}^2} \binom{u_1+u_2}{u_1} \, \v x^\v u.
\]  
Namely, the binomial coefficients satisfy
\[
\binom{(u_1+u_2)p^s}{u_1p^s} \is \binom{(u_1+u_2)p^{s-1}}{u_1p^{s-1}} \mod {p^{2s}}.
\]
This supercongruence is `twice stronger' than the Gauss' congruence in \S\ref{sec:gauss-cong}.  
\bigskip

\subsection{Calabi--Yau families}\label{sec:CY-families}
In~\cite[\S 7]{DCIII} we describe excellent Frobenius lifts for \emph{completely symmetric Calabi-Yau families}. A lattice polytope $\Delta \subset \R^n$ is called \emph{reflexive} if it is of maximal dimension and every codimension~1 face of $\Delta$ can be given by an equation $\sum_{i=1}^n a_i u_i=1$ with all $a_i \in \Z$. This condition implies that $\Delta^\circ \cap \Z^n =\{\v 0\}$ and every $\v u \in \Z^n$ lies on an integral dilation of the Euclidean boundary of $\Delta$. 

\begin{definition}\label{CY-family-def}\footnote{Reflexivity of the Newton polytope is closely connected to the Calabi--Yau property . Namely, Victor Batyrev showed in~\cite{Batyrev} that when $R=\C$, $f \in \C[x_1^{\pm1},\ldots,x_n^{\pm1}]$ has a reflexive Newton polytope $\Delta$ and satisfies a certain condition called $\Delta$-regularity, then the toric hypersurface $f(\v x)=0$ can be compactified to a Calabi--Yau variety in a certain projective space $\P_\Delta$ constructed using the Newton polytope only. 

A Calabi-Yau variety of dimension $d$ has a nowhere vanishing holomorphic differential $d$-form which is unique up to multiplication by a constant. In our setup this form on the hypersurface of zeroes $X_f$ is represented by the element $1/f(\v x)\in \Omega_f$.}A  Calabi--Yau family is given by 
\[
1-tg(\v x) = 0 
\]
with $g(\v x)\in\Z[x_1^{\pm1},\ldots,x_n^{\pm1}]$ whose Newton polytope $\Delta \subset \R^n$ is reflexive. A Calabi--Yau family is called completely symmetric if the only non-zero integral points in the Newton polytope are vertices and there is a finite subgroup $\sG \subset {\rm GL}_n(\Z)$ which acts transitively on the set of vertices and preserves $g(\v x)$. 
\end{definition}

Here are some examples. The reader will find their detailed disussion in~\cite[\S 7]{DCIII}:
\begin{itemize}
\item simplicial family
\[
g(\v x) = x_1+\ldots+x_n+\frac1{x_1\cdots x_n}
\]
\item hyperoctahedral family
\[
g(\v x) = x_1+\frac1{x_1}+\ldots+x_n+\frac1{x_n}
\]
\item hypercubic family
\[
g(\v x) = \left(x_1+\frac1{x_1}\right)\cdots\left(x_n+\frac1{x_n}\right)
\]
\item $A_n$-family
\[
g(\v x) = \left(1+x_1+\ldots+x_n\right)\left(1 + \frac1{x_1}+\ldots+\frac1{x_n}\right)
\]

\end{itemize}

One takes $f(\v x)=1-tg(\v x)$. For $R$ we take a $p$-adically closed subring of $\Z_p\lb t \rb$. More precisely, 
\be{CY-fil2-ring-R}
R = \text{ $p$-adic completion of } \Z[t,1/hw_1(t),1/hw_2(t)],
\ee
where $hw_1(t),hw_2(t)\in \Z[t]$ are the so-called first and second Hasse--Witt polynomials. These polynomials are defined in~\cite[\S 7]{DCIII} for a given Frobenius lift $\sigma$, but change of the lift results in a polynomial which equals to the previous one modulo $p$. In the view of (ii) in Exercise~\ref{frob-lift-exercise} the ring in~\eqref{CY-fil2-ring-R} is defined independently of the choice of $\sigma$. 

Let $\Gamma \subseteq \Z^n$ be the lattice spanned by the vertices of $\Delta$. For completely symmetric Calabi--Yau families for which $\Gamma=\Z^n$ one can show that 
\[
\Omega_f(\Delta^\circ)^\sG / \fil_2 \cong R  \frac{1}{f(\v x)} + R \left(\frac{1}{f^2(\v x)} \right).\footnote{If $\Gamma \subsetneq \Z^n$ then one also needs to restrict to rational functions whose numerators are supported in $\Gamma$. Such modules are examples of more sophisticated \emph{Dwork crystals} considered in~\cite{DCIII}. The theory of this section works for them along the same lines, but we prefer to avoid these cases here for simplicity. The reader may check that $\Gamma=\Z^n$ for simplicial, hyperoctahedral and $A_n$ families, while for the hypercubic family one has $[\Z^n:\Gamma]=2^{n-1}$.}
\]
Here $\Omega_f(\Delta^\circ)^\sG$ means the submodule of $\sG$-invariant elements in $\Omega_f(\Delta^\circ)$. 
Another basis in this rank 2 $R$-module is given by $1/f$ and $\theta(1/f)$, where $\theta= t \frac {d}{dt}$. Therefore for any Frobenius lift $\sigma: R \to R$ one has
\[
\cartier\left(\frac{1}{f(\v x)} \right) = \lambda_0(t) \frac{1}{f^\sigma(\v x)} + \lambda_1(t) \left( \theta \frac{1}{f(\v x)} \right)^\sigma \mod {p^2 \fil_2^\sigma}.
\]
Here elements $\lambda_0,\lambda_1 \in R$ depend on $\sigma$ and can be determined explicitly in terms of the Picard--Fuchs differential equation satisfied by the period function 
\[\bal
F_0(t) &= \frac1{(2 \pi i)^n} \oint \cdots \oint \frac1{1-t g(\v x)} \frac{dx_1}{x_1} \cdots \frac{dx_n}{x_n}\\
& = \sum_{m=0}^{\infty} c_m t^m, \quad c_m = \text{ constant term of } g(\v x)^m. 
\eal\]
Earlier we denoted the series $F_0(t)$ by $\gamma(t)$, but a change of notation at this point is useful. 
Denote $L_2 = \theta^2 + A(t) \theta + B(t)$ where $A,B \in R$ are the elements determined by the condition $L_2(1/f) \in \fil_2$. Note that $F_0(t)=\sP_{\v 0}(1/f(\v x))$ where $\sP_{\v 0}: \Omega_f \to \Z_p\lb t \rb$ is the period map of Example~\ref{period0example}.  Since this period map commutes with the derivation $\theta$, series $F_0(t) = \sP_{\v 0}(1/f)$ is annihilated by the differential operator $L_2$. There is unique second solution to $L_2 y = 0$ which has the form
\[
\log(t) F_0(t) + F_1(t), \quad F_1(t) \in t\Q[[t]]. \footnote{In the Appendix to~\cite{DCIII} it is shown that the series $F_1(t)$, and hence also $A(t),B(t) \in R$ do not depend on the prime $p$. There is a Picard--Fuchs differential operator, which is an operator of minimal degree $L \in \Q(t)[\theta]$ such that $L(1/f(\v x)) \in d\Omega_f$. This $L$ annihilates the series $F_0(t)$ (obviously) and $\log(t) F_0(t) + F_1(t)$ (much less obviously), and therefore $L_2$ is a right factor of $L$ in the bigger ring $\Q[[t]][\theta]$. Operator $L$ describes the Gauss-Manin connection on the de Rham module $\Omega_f(\Delta^\circ)^\sG / d \Omega_f$. It will appear in \S\ref{sec:frob-structure}. }
\]
In~\cite[\S 7]{DCIII} we compute $\lambda_0$ and $\lambda_1$ in terms of the series $F_0(t)$, $F_1(t)$ and the Frobenius lift $\sigma$. We find that there is a unique excellent lift $\sigma_0$. It can be conveniently described in terms of the so-called \emph{canonical coordinate}
\[
q(t) = \exp \left(\frac{\log(t) F_0(t) + F_1(t)}{F_0(t)}\right) = t \exp \left(\frac{F_1(t)}{F_0(t)}\right) \in t + t^2 \Q[[t]]. 
\]

\begin{theorem}{\cite[Theorem 7.3]{DCIII}}\label{exc-lift-thm} Consider a completely symmetric Calabi--Yau family 
\[
1 - t g(\v x) = 0
\]
as in Definition~\ref{CY-family-def}. Assume that $p \nmid \# \sG \times [\Z^n:\Gamma] \times c$, where $\sG$ is the symmetry group of the family, $\Gamma$ is the lattice generated by the vertices of $\Delta$ and $c \in \Z$ is the vertex coefficient of $g(\v x)$. Then 
\begin{itemize}
\item[(i)] the canonical coordinate is $p$-integral: $q(t) \in t + t^2 \Z_p\lb t \rb$;
\item[(ii)] there is a unique Frobenius lift $\sigma_0:\Z_p\lb t \rb \to \Z_p\lb t \rb$ such that
\[
\cartier\left(\frac{1}{f(\v x)} \right) = \lambda(t) \frac{1}{f^{\sigma_0}(\v x)} \mod {p^2 \fil_2^{\sigma_0}} \quad \text{ with } \lambda(t)=\frac{F(t)}{F(t^{\sigma_0})};
\]
\item[(iii)] the~\emph{excellent lift} $\sigma_0$ is given in terms of the canonical coordinate by
\[
q \mapsto c^{p-1} q^p;
\]
moreover, one has $\sigma_0(R) \subset R$ where $R$ is the ring defined in~\eqref{CY-fil2-ring-R}.
\end{itemize}
\end{theorem}

The inverse series $t=t(q)\in q+q^2 \Q[[q]]$ is called the~\emph{mirror map}. One particular consequence of this theorem is that the series $t^{\sigma_0}=\sigma_0(t)$ belongs to the ring~\eqref{CY-fil2-ring-R}. This ring is contained in the \emph{field of $p$-adic analytic elements}
\be{p-adic-analytic-elts}
E_p = \text{ $p$-adic completion of } \Q(t). 
\ee
It is often the case for Calabi-Yau families in $n \le 3$ dimensions that $t(q)$ is an elliptic modular function (see~\cite[\S 8]{DCIII}). In~\cite[\S 7]{Dw-p-adic-cycles} Dwork proved that for the elliptic $j$-invariant, the series $j(q^p)$ is a $p$-adic analytic function of $j(q)$. Particularly, he shows it can be approximated by rational functions of $j$ whose denominators are powers of a first Hasse--Witt polynomial which we denote by $hw^{(1)}(t)$ in our setup. This is a polynomial whose roots $\mod p$ are precisely the supersingular $j$-invariants. Dwork attributes this theorem to Deligne. For $n \ge 4$ we do not expect any modularity of $t(q)$. Still, our result shows that $t(q^p)$ is a $p$-adic analytic function of $t(q)$, though in general the second Hasse--Witt polynomial also occurs in denominators of interpolating rational functions. We shall look at fixed points of these $p$-adic analytic functions.

\begin{exercise} In the setup of Theorem~\ref{exc-lift-thm}, prove that for every $a \in \F_p$ such that $hw^{(1)}(a) \ne 0$ and $hw^{(2)}(a) \ne 0$ there is a unique $\tilde a \in \Z_p$, $\tilde a \is a \mod p$ such that $\tilde a$ is a fixed point of the excellent Frobenius lift: 
\[
t^{\sigma_0}(\tilde a) = \tilde a.
\] 
\end{exercise}

\begin{problem} Compute fixed points of the excellent Frobenius lift for the specific completely symmetric Calabi--Yau families given earlier in this section.  
\end{problem}

This problem is of interest because we expect supercongruences to hold at these points as we demonstrated in a simple example in the end of~\S\ref{sec:exc-lifts}. 

\subsection{Cartier matrices on de Rham quotients}

To shorten the notation, we write $\Omega_f(\mu)/d \Omega_f$ for what actually is $\Omega_f(\mu)/(d \Omega_f \cap \Omega_f(\mu))$. In situations when this is a free $R$-module of finite rank one can combine reduction modulo derivatives with computation of the Cartier action modulo higher formal derivatives by means of congruences in order to obtain Cartier matrices on $\Omega_f(\mu)/d \Omega_f$. Let us list the necessary steps:
\begin{itemize}
\item determine a basis in  $\Omega_f(\mu)/d \Omega_f$,
\item determine $m$ such that $\fil_m \subset d\Omega_f \cap \Omega_f(\mu)$, check the $m$th Hasse--Witt condition for $\mu$ and solve systems of congruences to determine the matrix $\Lambda^{(m)}(\mu)$ of the Cartier action modulo $\fil_m$, 
\item describe reduction of the basis in $\Omega_f^{(m)}(\mu)$ from the second step to the basis in $\Omega_f(\mu)/d \Omega_f$ which was found in the first step and compute the matrix of the Cartier action modulo $d\Omega_f$ using $\Lambda^{(m)}(\mu)$. 
\end{itemize}

In~\cite{Cartier0} we performed these steps for \emph{simplicial} and \emph{hyperoctehedral} Calabi-Yau families in $n$ dimensions. The results will be presented in \S\ref{sec:frob-structure}. Let us note here that checking higher Hasse--Witt conditions may be a difficult task. For these $n$-dimensional families we needed the $n$th Hasse--Witt condition. The following theorem shows that these conditions hold rather generally for families whose Newton polytopes are sufficiently simple. We consider families 
\[
1-t g(\v x) = 0
\]
with $g \in \Z[x_1^{\pm 1},\ldots,x_n^{\pm 1}]$ and assume that the Newton polytope $\Delta$ is reflexive. Examples were given in \S\ref{sec:CY-families}.  A proper face $\tau \subseteq \Delta$ is called \emph{a simplex of volume~$1$} if it has $\dim(\tau)+1$ vertices and all lattice points in the $\R_{\ge0}$-cone generated by $\tau$ are integer linear combinations of the vertex vectors. More generally, the simplicial volume of $\tau$ is the index of the lattice of points generated as $\Z$-linear combinations of its vertices in the lattice of integral points in the $\R$-vector space spanned by them. 

\begin{theorem}{\cite[\S 3]{IN}}
Suppose $p>n$. If all proper faces of $\Delta $ are simplices of volume~1 and all vertex coefficients of $g(\v x)$ are in $\Z_p^\times$ then for any open $\mu \subseteq \Delta$ and any $k<p$, the $k$-th Hasse--Witt condition holds for $\mu$ over the ring $\Z_p\lb t \rb$.  
\end{theorem} 

\begin{problem} State assumptions under which Hasse--Witt conditions hold for polytopes of more general shape. 
\end{problem}

Solutions to this problem would allow to apply our methods to a bigger realm of Calabi--Yau families.

\subsection{$p$-adic Frobenius structures on differential equations}\label{sec:frob-structure}

Let us look at the situation when the coefficients of $f(\v x)$ depend on a parameter, that is we have $\Z[t] \subset R \subset \Z_p\lb t \rb$ as in examples of Calabi--Yau families from \S\ref{sec:CY-families}. In these case Cartier matrices, which we usually denoted by $\Lambda$, also depend on $t$. We will show that they satisfy very particular differential equations. 

Suppose that $\Omega_f(\mu)/d\Omega_f$ is a free $R$-module of rank $m$ and $\Lambda(t)=(\lambda_{i,j}(t))$ is the matrix of Cartier operation in some basis $\omega_1,\ldots,\omega_m$ of this quotient module:
\be{cartier-matrix-frob-str}
\cartier \omega_i = \sum_{j=1}^m \lambda_{i,j}(t) \omega_j^\sigma \quad \mod {d\hat\Omega_f}
\ee
for $i=1,\ldots,m$. We assume that $R$ is $p$-adically complete and $\hat\Omega_f(\mu)/d\hat\Omega_f = \Omega_f(\mu)/d\Omega_f$. This will be the case in the examples below.  Note that derivations of $R$, which act on rational functions in $\Omega_f(\mu)$ by the usual rules of differential calculus, preserve this modules. Since they commute with derivations in the variables $x_1,\ldots,x_n$, derivations of $R$ also map the submodule $d\Omega_f$ to itself. This fact turns $\Omega_f(\mu)/d\Omega_f$ into a differential module. Pick a derivation, e.g. $\theta=t \frac{d}{dt}$, and write down the matrix of its action in our basis: let $N \in R^{m \times m}$ be such that
\be{connection-matrix}
\theta\omega_i = \sum_{j=1}^m N_{ij}(t) \omega_j \quad \mod {d\Omega_f}.
\ee
Let $\widetilde N(t)$ denote the matrix of $\theta$ in the basis $\omega_i^\sigma$, $i=1,\ldots,m$ in $\Omega_{f^\sigma}(\mu)/d\Omega_{f^\sigma}$.   

\begin{proposition} The Cartier matrix $\Lambda$ satisfies the differential equation 
\be{frob-de}
\theta(\Lambda) = N \Lambda - \Lambda \widetilde N.
\ee
\end{proposition}  
\begin{proof} Note that $\cartier$ and $\theta$ commute as operations on $\hat\Omega_f$ and $\hat\Omega_{f^\sigma}$: 
\[
\theta \circ \cartier = \cartier \circ \theta.
\] 
Perhaps the easiest way to check that they commute is by looking on the Cartier action on formal expansions. Let us apply $\theta$ to the equation~\eqref{cartier-matrix-frob-str}:
\[
\theta(\cartier \, \omega_i) = \sum_{j=1}^m \left( \theta(\lambda_{i,j}) \omega_j^\sigma + \lambda_{i,j}\sum_{k=1}^m \widetilde N_{j,k} \omega_k^\sigma \right)\quad \mod {d\hat\Omega_{f^\sigma}}. 
\] 
On the other hand, the Cartier operation is $R$-linear and therefore
\[
\cartier(\theta \omega_i) = \cartier\left(  \sum_{j=1}^m N_{i,j} \omega_j\right) = \sum_{j=1}^m N_{i,j} \sum_{j=1}^m \lambda_{j,k}\omega_k^\sigma \quad \mod {d\Omega_{f^\sigma}}.
\]
Since $\cartier$ and $\theta$ commute, the two above expressions are equal and we obtain that $\theta(\Lambda) + \Lambda \widetilde N = N \Lambda$. This is precisely our claim.  
\end{proof}

The situation of modules $\Omega_f(\mu)/d^k \Omega_{\rm formal}$ with $k \ge 1$ and matrices $\Lambda^{(k)}(t)$ can be considered in exactly the same way yielding the same differential equation~\eqref{frob-de}, see~\cite[Prop 5.12]{DCIII}.

\begin{exercise} Suppose $U(t)$, $V(t)$ are fundamental matrices of solutions to $\theta(U)=N U$ and $\theta(V)=\widetilde N V$ respectively. Assuming the multiple $U^{-1} \Lambda V$ makes sense, use~\eqref{frob-de} to check that $\theta(U^{-1} \Lambda V)=0$. We can conclude that 
\[
\Lambda(t) = U(t) \Lambda_0 V(t)^{-1},
\]
where $\Lambda_0$ is a constant matrix. 
\end{exercise}

Let us now suppose that the differential module $\Omega_f(\mu)/d\Omega_f$ has a \emph{cyclic basis}, that is we have an element $\omega \in \Omega_f(\mu)$ such that $\theta^i(\omega)$, $i=0,\ldots,m-1$ is a basis in this quotient module. Let 
\[
L = \theta^m + a_1(t) \theta^{m-1} + \ldots + a_{m-1}(t) \theta + a_m(t) \in R[\theta]
\]
be the differential operator such that $L \omega \in d\Omega_f$.  We will call it the Picard--Fuchs differential operator. Then 
\[
N=\begin{pmatrix}0&1&0&\ldots &0 \\0&0&1&\ldots&0\\&\vdots&\\-a_m(t)&-a_{m-1}(t)&&\ldots&-a_1(t)\end{pmatrix}.
\]
Let $y_0,\ldots,y_m$ be a basis of solutions to the differential equation $Ly=0$. Then it is not hard to see that $U(t)=(\theta^iy_j)_{0 \le i,j \le m-1}$ is a fundamental solution to the system $\theta(U)=NU$. The Cartier matrix in the bases $\theta^i \omega$ and $(\theta^i \omega)^\sigma$ repectively is then given by
\[
\Lambda(t) = U(t) \Lambda_0 U(t^\sigma)^{-1},
\]
where the constant matrix $\Lambda_0$ is yet to be determined. This formula expresses the fundamental observation due to Dwork that zeta functions in families can be computed using solutions to their Picard--Fuchs differential equations.

\begin{definition}
After Dwork, a \emph{$p$-adic Frobenius structure} for an ordinary differential operator 
\[
L = \theta^m + a_1(t) \theta^{m-1} + \ldots + a_{m-1}(t) \theta + a_m(t) \in \Q(t)[\theta]
\]
is defined as an invertible matrix $\Lambda(t) \in E_p^{m \times m}$ satisfying the differential equation
\[
\theta \Lambda(t) = N(t) \Lambda(t) - p \, t^p \Lambda(t) N(t^p).
\]
\end{definition}

Here $E_p$ is the field of $p$-adic analytic elements, see~\eqref{p-adic-analytic-elts}. For simplicity we have given this definition for the Frobenius lift $t^\sigma = t^p$, but it is not hard to modify it for general Frobenius lifts $\sigma$. The differential equation here is the same as~\eqref{frob-de}, if we note that $V(t)=U(t^p)$ satisfies the system $\theta V = \tilde N V$ with $\tilde N(t) = p t^p N(t^p)$. In~\cite{Dw-uniqueness} Dwork proves that, if a $p$-adic Frobenius structure exists for an irreducible operator $L$, then it is unique up to multiplication by a non-zero $p$-adic constant. 

We can summarize that our Cartier matrices provide examples of $p$-adic Frobenius structures for families of hypersurfaces. Let us now give examples in which the Cartier matrices are given for almost all but finitely many primes $p$ by a sort of universal expression, that is $p$ enters the formula for $\Lambda(t)$ as a variable.

\begin{example}\label{simplicial-de-example} Consider the simplicial family
\[
1 - t \left(x_1+x_2+\ldots+x_n + \frac1{x_1\ldots x_n}\right) = 0.
\]
Denote $s_n(t)=(n+1)(1-((n+1)t)^{n+1})$. Over the ring 
\[
R = \Z[t,s_n(t)^{-1}]
\] 
one has
\[
\Omega_f(\Delta^\circ) / d \Omega_f \cong \oplus_{i=0}^{n-1} R \, \theta^i \left(\frac1{f(\v x)}\right).
\]
The respective Picard--Fuchs operator $L$ such that $L(1/f(\v x))\in d\Omega_f$ is given by
\be{simplicial-PF}
L = \theta^n - ((n+1)t)^{n+1}(\theta+1)\ldots (\theta+n).
\ee
The proof of this fact can be found in~\cite[\S 5]{IN}.
\end{example}

The Picard--Fuchs differential operator~\eqref{simplicial-PF} in the simplicial example has maximal unipotent local monodromy at $t=0$. In this case one can define the \emph{standard basis} of solutions to $L$ near $t=0$. This is a unique basis of the form
\be{st-basis}\bal
y_0(t) &= F_0(t) \in 1+t\Q\lb t \rb\\
y_1(t) &= F_0(t) \log(t) + F_1(t)\\
y_2(t) &= F_0(t) \frac{\log^2(t)}{2!} + F_1(t) \log(t) + F_2(t)\\
\ldots\\
y_{n-1}(t) &= \sum_{j=0}^{n-1} F_j(t) \frac{\log^{n-1-j}(t)}{(n-1-j)!} 
\eal\ee
where $F_j(t)\in \Q\lb t \rb$, $F_0(0)=1$ and $F_i(0)=0$ for $i>1$. Now we can present Cartier matrices for simplicial families with respect to the cyclic bases $\theta^i(1/f(\v x))$ and $\theta^i(1/f(\v x))^\sigma$. For simplicity we take the Frobenius lift $t^\sigma = t^p$.

\begin{theorem}[\S 5 of \cite{IN} and Theorem 1.4 in \cite{Cartier0}]\label{Cartier0main} For every $p>n+1$ the Cartier matrix for the simplicial family in $n$ dimensions is given by $\Lambda(t) = U(t) \Lambda_0 U(t^p)^{-1}$ where $U(t)=(\theta^iy_j)_{0 \le i,j \le n-1}$ is the Wronskian matrix of standard solutions~\eqref{st-basis} to the Picard--Fuchs operator~\eqref{simplicial-PF} and
\[
\Lambda_0 = \Lambda(0) = \begin{pmatrix}1&p \alpha_1&p^2 \alpha_2&\ldots &p^{n-1} \alpha_{n-1} \\0&p&p^2 \alpha_1&\ldots&p^{n-1}\alpha_{n-2}\\&\vdots&\\0&0&&\ldots&p^{n-1}\end{pmatrix}
\] 
with 
\[
\alpha_j = \text{ coefficient of $x^j$ in } \frac{\Gamma_p(x)}{\Gamma_p(x/(n+1))^{n+1}}, \; j=1,\ldots,n-1.
\] 
Entries of the Cartier matrix $\Lambda(t)=(\lambda_{i,j}(t))_{0 \le i,j\le n-1}$ satisfy
\[
\lambda_{i,j}(t) \in p^j R,
\]
where $R \subset \Z_p\lb t \rb$ is the $p$-adic completion of $\Z[t,s_n(t)^{-1}]$, where $s_n(t)=(n+1)(1-((n+1)t)^{n+1})$.
\end{theorem}

This theorem provides an evidence to the prediction of Candelas, de la Ossa and van Straten that $p$-adic zeta values occur as entries of the Frobenius structure in Calabi--Yau families at $t=0$. More precisely, in~\cite[\S 4.4]{COS21} they conjecture that in general, for Calabi-Yau families with a point of maximal unipotent monodromy at $t=0$ the constants $\alpha_j$ are $\Q$-linear combinations of $p$-adic zeta values and their products. Moreover, the rational coefficients of these linear combinations are universal (independent of $p$). This is true in the case of simplicial families because of the following expansion of the logarithm of the Morita $p$-adic gamma function at $x=0$: 
\[
\log \Gamma_p(x)=\Gamma'_p(0)x-\sum_{m\ge2}\frac{\zeta_p(m)}{m}x^m.
\] 
Here all even $p$-adic zeta values are zero: $\zeta_p(2m)=0$ for $m \ge 1$. With this we find that
\[\bal
\frac{\Gamma_p(x)}{\Gamma_p(x/5)^5} &= 1 - \frac{8}{25} \zeta_p(3)  x^3 + O(x^4) \\
\frac{\Gamma_p(x)}{\Gamma_p(x/6)^6} &= 1 - \frac{35}{108} \zeta_p(3) x^3 + O(x^5) \\
\frac{\Gamma_p(x)}{\Gamma_p(x/7)^7} = 1 - &\frac{16}{49} \zeta_p(3) x^3 - \frac{480}{2401} \zeta_p(5) x^5 + O(x^6) \\
\frac{\Gamma_p(x)}{\Gamma_p(x/8)^8} = 1 - \frac{21}{64} \zeta_p(3) & x^3 - \frac{819}{4096} \zeta_p(5) x^5 + 
\frac{441}{8192} \zeta_p(3)^2 x^6 + O(x^7) \\
\eal\]
We also note that $\alpha_1=\alpha_2=0$ for any $n$.

\begin{example}\label{hyperoctahedral-Cartier0} In~\cite[\S 6]{IN} and \cite[Theorem 1.5]{Cartier0} we prove a similar theorem for hyperoctahedral Calabi--Yau family in $n$ dimensions. In this case also $\alpha_1=\alpha_2=0$, and in general\footnote{In our paper we have another somewhat more sophisticated expression for $\alpha_j$'s. The simplification given here was pointed out to us by Don Zagier.}
\[
\alpha_j = \text{ coefficient of $x^j$ in } e^{-\Gamma_p'(0)x}\Gamma_p(x), \; j=1,\ldots,n-1.
\]
In contrast to the simplicial example above, hyperoctahedral Picard--Fuchs differential operators aren't hypergeometric. 
\end{example}
 
\subsection{Integrality of mirror maps and instanton numbers}

\emph{Mirror symmetry} is a relationship between Calabi--Yau manifolds. Two such manifolds may be very different geometrically but are nevertheless equivalent when employed as `extra dimensions' to describe interaction of particles in string theory. In~\cite{COGP1991} physicists Candelas, de la Ossa, Green and Parkes made a sensation: they predicted solutions to very hard problems of enumerative geometry by doing miraculous computations with solutions of a differential equation 'on the other side of the mirror'. We will briefly recall their discovery.  

Candelas, de la Ossa, Green and Parkes considered the differential operator
\be{quintic-operator}
L = \theta^4 - 5^5 t \left(\theta+\tfrac15\right)\left(\theta+\tfrac25\right)\left(\theta+\tfrac35\right)\left(\theta+\tfrac45\right)
\ee
where $\theta=t \tfrac{d}{dt}$. The differential equation $Ly =0$ has a singular point at $t=0$ with maximally unipotent local monodromy. There is a unique (up to multiplication by a constant) solution which is analytic analytic at $t=0$:
\[
y_0(t) = \sum_{n=0}^{\infty} \frac{(5n)!}{n!^5} t^n = 1 + 120 t + 113400 t^2 + \ldots = F_0(t) \in \Z\lb t \rb.
\]
There is also a unique solution of the form
\[
y_1(t) = F_0(t) \log(t) + F_1(t) \quad \text{ with } F_1(t) \in t\Q\lb t \rb.
\]
Here  $F_1(t) = \sum_{n=1}^\infty  \frac{(5n)!}{n!^5}\left(\sum_{j=1}^{5k} \frac5j\right) t^k$ does not have integral coefficients in contrast to $y_0(t)$, but the authors of~\cite{COGP1991} observe integrality of the \emph{canonical coordinate}:
\[
q(t) = \exp\left( \frac{y_1(t)}{y_0(t)}\right) = t \exp\left( \frac{f_1(t)}{f_0(t)}\right) \in t\Z\lb t \rb.
\] 
This observation was proved by by B.-H.Lian and S.-T.Yau in 1996. \footnote{If we substitute $t \mapsto t^5$ in $L$ we will obtain the Picard--Fuchs operator of the simplicial family with $n=4$, see Example~\ref{simplicial-de-example}. It the follows from (i) in our Theorem~\ref{exc-lift-thm} that the canonical coordinate is $p$-integral for every $p>5$.} We shall also consider the next solution of the form
\[
y_2(t) = F_0\frac{\log(t)^2}{2!} + F_1 \log(t) + F_2 \quad \text{ with }  F_2 \in t \Q\lb t \rb
\] 
and express the ratios $y_i/y_0$ in terms of the canonical coordinate:  
\[\bal
&\frac{y_0}{y_0}=1, \quad \frac{y_1}{y_0}=\log(q), \\
&\frac{y_2}{y_0}=\frac12\log(q)^2 + 575 q + \frac{975375}{4} q^2 +    + \frac{1712915000}{9} q^3 + \ldots
\eal\] 
The series
\[
Y(q) = \left( q \frac{d}{dq}\right)^2 \frac{y_2}{y_0} = 1 + 575 q + 975375 q^2 + \ldots
\]
is called the \emph{Yukawa coupling}. Candelas, de la Ossa, Green and Parkes write its Lambert expansion in the form
\[
Y(q) = 1 + \sum_{d \ge 1} N_d \, d^3 \, \frac{q^d}{1-q^d}
\]
and call $N_d$ \emph{instanton numbers}. They predict that for every $d \ge 1$ numbers $5 N_d$ coincide with the numbers of degree $d$ rational curves that lie on a generic threefold of degree $5$ in $\mathbb{P}^4$:
\[\bal
&5 N_1 = 2875, \quad 5 N_2 = 609250, \\
&5 N_3 = 317206375, \quad 5 N_4 = 242467530000, \quad ...
\eal\]
Only the first two numbers were known at that time! The number $2875$ of lines on a general quintic was determined by H. Schubert in 1886. The number $609250$ of conics was determined  by S. Katz in 1986. In 1993 G.Ellingsrud and S.Str\o mme computed the number of cubic curves on the quintic threefold. Their result served as a crucial cross-check for the above physicists' prediction, which was made in 1991.  

In 1990s the Gromov--Witten theory was developed to provide a rigorous basis for counting curves on general manifolds. Subsequently, Givental and Lian--Liu--Yau proved the \emph{mirror theorem} which justified the equality of instanton numbers and genus zero Gromov--Witten invariants. Let us note that, in contrast to the numbers of curves of given degree on a manifold, instanton numbers can be computed  easily. Computations showed that in the above case~\eqref{quintic-operator}, known as the \emph{quintic case}, instanton numbers are integral, though a priori they are expected to be rational numbers. Similarly, rigorous count of numbers of rational curves of given degree (Gromov--Witten invariants) yields a priory rational numbers, as this number is obtained via integration over a moduli space of such curves. In~\cite{IN} we prove the following claim.

\begin{theorem} [\cite{IN}, Corollary 1.9]\label{quntic-case-thm} For the quintic differential operator~\eqref{quintic-operator} instanton numbers are $p$-integral for every $p > 5$.  
\end{theorem}  

Let us remark that Gromov--Witten invariants are a special case of more general BPS-numbers, whose integrality was proved by Ionel and Parker in~\cite{BPS}. Combination of their result with the mirror theorem should also yield a proof of integrality of instanton numbers. The advantage of our proof is that it does not use mirror symmetry at all, it is a direct proof in terms of the differential equation itself. 
Namely, we exploit the fact that operator $L$ has a $p$-adic Frobenius structure with special properties for almost all primes $p$. We will sketch the idea in the rest of this section.

Consider a 4th order differential operator 
\be{order-4-operator}
L = \theta^4 + a_1(t)\theta^3 + a_2(t) \theta^2 + a_3(t) \theta + a_4(t)
\ee
with $a_i \in \Q(t)$ for $i=1,\ldots,4$. Assume also that $a_i(0)=0$ for all $i$. This is the condition under which $L$ has a maximal unipotent monodromy at $t=0$. Consider the standard basis of solutions to the differential equation $Ly=0$ near $t=0$:
\[
y_i(t) = \sum_{j=0}^{i} F_j(t) \frac{\log^{i-j}(t)}{(i-j)!}, \quad i=0,1,2,3
\]
with $F_0 \in 1+t\Q\lb t \rb$ and $F_j \in t\Q\lb t \rb$ for $j>0$. Consider the fundamental matrix of solutions $U(t)=(\theta^i y_j)_{0 \le i,j \le 3}$. As we explained in \S\ref{sec:frob-structure}, a $p$-adic Frobenius structure for the operator $L$ is a $4 \times 4$ matrix 
\be{frob-str-matrix-rank-4}
\Lambda(t) = (\lambda_{ij}(t))_{0 \le i,j \le 3} = U(t) \begin{pmatrix}1&p \alpha_1&p^2 \alpha_2&p^3 \alpha_3 \\
0&p&p^2 \alpha_1&p^3\alpha_2\\
0&0&p^2&p^3\alpha_1\\
0&0&0&p^{3}\end{pmatrix} U(t^p)^{-1}
\ee
with special $p$-adic numbers $\alpha_1,\alpha_2,\alpha_3$ such that the matrix entries $\lambda_{ij}(t)$ are $p$-adic analytic elements. That is, they can be approximated by $p$-adically by elements of $\Q(t)$. Existence of such matrix is a strong arithmetic property of the operator $L$, and if it exists it is unique. For the questions of integrality in this section we are interested in matrices $\Lambda(t)$ as above with the property that
\be{CY-frob-str}
\lambda_{ij}(t) \in p^j \Z_p\lb t \rb \quad \text{ for all } \quad 0 \le i,j \le 3. 
\ee

\begin{theorem}[\cite{IN}, Theorems 1.4, 1.6 and 1.7]\label{IN-first-theorems} Suppose for the operator $L$ as in~\eqref{order-4-operator} there exist $p$-adic numbers $\alpha_1,\alpha_2,\alpha_3$ such that entries of matrix~\eqref{frob-str-matrix-rank-4} have property~\eqref{CY-frob-str}. Then
\begin{itemize}
\item[(i)] the holomorphic solution is $p$-integral: $y_0 \in \Z_p \lb t \rb$,
\item[(i)] the canonical coordinate is $p$-integral: $q = \exp(y_1/y_0) \in \Z_p \lb t \rb$,
\item[(iii)] if in addition $L$ is self-dual and $\alpha_1=0$, then the instanton numbers of $L$ are $p$-integral: $n_d \in \Z_p$ for all $d \ge 1$
\end{itemize}
\end{theorem}       

In Theorem~\ref{Cartier0main} and Example~\ref{hyperoctahedral-Cartier0} we have seen that the Cartier matrices $\Lambda(t)$ have $\alpha_1=0$ for simplicial and hyperoctahedral families with $n=4$. Condition~\eqref{CY-frob-str} is not very hard to check in our construction of the Cartier operator, see~\cite[Prop 4.2]{IN}. The above theorem then implies $p$-integrality for almost all $p$ of instanton numbers for the Picard--Fuchs operators 
\[
L = \theta^4 - (5 t)^5 \left(\theta+1\right)\left(\theta+2\right)\left(\theta+3\right)\left(\theta+4\right)
\]
and 
\[\bal
L = (1024t^4& - 80t^2 + 1)\theta^4 + 64(128t^4 - 5t^2 )\theta^3 \\
&+16( 1472 t^4 - 33t^2 )\theta^2 + 32(896t^4 - 13t^2 )\theta + 128(96t^4 - t^2)
\eal\]
of the simplicial and hyperoctahedral families respectively. The quintic operator~\eqref{quintic-operator} differs from the first operator above by the change of variables $t \mapsto t^5$. One can show that  substitutions $t \mapsto t^N$ preserve $p$-integrality of canonical coordinates and instanton numbers for all $p \nmid N$, and hence Theorem~\ref{quntic-case-thm} also follows from the simplicial case.  

\begin{problem} In~\cite{AESZ}  Gert Almkvist, Christian van Enckevort, Duco van Straten and Wadim Zudilin make a search among differential operators $L$ which are of order 4, self-dual and have a point of maximally unipotent monodromy at $t=0$. They look for operators which have arithmetic properties (i)-(iii) as in Theorem~\ref{IN-first-theorems} for almost all primes $p$. Such operators are called Calabi-Yau differential operators, and more than 400 such operators we found \emph{experimentally} up to now. Property (i) usually holds for all known cases, (ii) was proved for some of them, and (iii) is only proved for the simplicial and hyperoctahedral operators. 
     
Construct $p$-adic Frobenius structures for other `experimental' Calalbi-Yau operators and check if they satisfy~\eqref{CY-frob-str} and have $\alpha_1=0$.    
\end{problem}

It seems natural to conjecture that all experimentally found Calabi--Yau operators have a $p$-adic Frobenius structure with property~\eqref{CY-frob-str} and $\alpha_1=0$ for almost all primes $p$, from which their arithmetic properties follow by Theorem~\ref{IN-first-theorems}.

\end{document}